\documentclass[12pt,thmsa, reqno]{amsart}

\usepackage{amssymb, euscript}

\def\th{\theta}

\def\Cal{\mathcal}

\def\P{{\Cal P}}

\def\I{{\Cal I}}

\def\f0{f_0}
\def\Fc0{\varphi_0}
\def\rn{\bbr^n}

\def\I_k {I_{-}^{k/2}}
\def\I+k {I_{+}^{k/2}}

\def\bbr{{\Bbb R}}

\def\bbn{{\Bbb N}}
\def\bbh{{\Bbb H}}

\def\bbc{{\Bbb C}}

\def\bbe{{\Bbb E}}
\def \hn{{\bbh^n}}

\def\supp{{\hbox{\rm supp}}}
\def\cosh{{\hbox{\rm cosh}}}
\def\sinh{{\hbox{\rm sinh}}}

\def\const{{\hbox{\rm const}}}
\def\cos{{\hbox{\rm cos}}}

\def\ch{{\hbox{\rm cosh}}}
\def\sh{{\hbox{\rm sinh}}}

\def\part{\partial}
\def\intl{\int\limits}
\def\b{\beta}

\def\Gam{\Gamma}
\def\Om{\Omega}
\def\a{\alpha}
\def\om{\omega}

\def\Del{\Delta}
\def\del{\delta}
\def\vp{\varphi}

\def\gam{\gamma}

\def\sig{\sigma}
\def\lam{\lambda}
\def\z{\zeta}
\def\e{\varepsilon}
\def\t{\tau}

\def\sn{S^{n-1}}

\font\frak=eufm10

\def\fr#1{\hbox{\frak #1}}

\def\frH{\fr{H}}

\def\const{{\hbox{\rm const}}}
\def\cos{{\hbox{\rm cos}}}

\def\part{\partial}
\def\intl{\int\limits}
\def\b{\beta}

\def\Gam{\Gamma}
\def\Om{\Omega}
\def\a{\alpha}

\def\hn{\bbh^n}

\newtheorem{theorem}{Theorem}[section]
\newtheorem{lemma}[theorem]{Lemma}

\newtheorem{corollary}[theorem]{Corollary}
\newtheorem{proposition}[theorem]{Proposition}

\theoremstyle{remark}
\newtheorem{remark}[theorem]{Remark}

\newtheorem{example}[theorem]{Example}

\numberwithin{equation}{section}

\newcommand{\be}{\begin{equation}}
\newcommand{\ee}{\end{equation}}

\newcommand{\bea}{\begin{eqnarray}}
\newcommand{\eea}{\end{eqnarray}}
\newcommand{\Bea}{\begin{eqnarray*}}
\newcommand{\Eea}{\end{eqnarray*}}

\def\sideremark#1{\ifvmode\leavevmode\fi\vadjust{\vbox to0pt{\vss% the remark
 \hbox to 0pt{\hskip\hsize\hskip1em%                          will appear only
\vbox{\hsize2cm\tiny\raggedright\pretolerance10000%          on the side
 \noindent #1\hfill}\hss}\vbox to8pt{\vfil}\vss}}}%
\begin{document}

\title[New Inversion Formulas for the Horospherical Transform]
{New Inversion Formulas for the Horospherical Transform}

\author{  B. Rubin }
\address{Department of Mathematics, Louisiana State University, Baton Rouge,
Louisiana 70803, USA}
\email{borisr@math.lsu.edu}

%\thanks{ The work  was supported in part by
%the Edmund Landau Center for Research in Mathematical Analysis
%and Related Areas, sponsored by the Minerva Foundation (Germany).}


%\date{November 26, 2001 and, in revised form 29.03.02.}

%\dedicatory{This paper is dedicated to our authors.}


\begin{abstract}

The  following two inversion methods   for  Radon-like transforms  are widely used in integral geometry and related harmonic analysis. The first  method invokes mean value operators in accordance with the classical Funk-Radon-Helgason scheme. The second one employs integrals of the potential type and polynomials of the Beltrami-Laplace operator. Applicability of these methods to the horospherical  transform in the  hyperbolic space $\hn$ was an open problem. In the present paper we solve this problem for $L^p$ functions in the maximal range of the parameter $p$ and for compactly supported smooth functions, respectively.  The main tools are harmonic analysis  on $\hn$ and associated fractional integrals.
 \end{abstract}

\maketitle

\section{Introduction}
\setcounter{equation}{0}

Let $\bbh^n$ be  the $n$-dimensional real hyperbolic space.
Several isometric models of $\bbh^n$ are known \cite{CFKP}. We will be dealing
 with the hyperboloid model,  when the space $\hn$ is identified with the upper sheet of the two-sheeted hyperboloid in the pseudo-Euclidean space $E^{n, 1} \sim \bbr^{n+1}$.

There are two different analogues of the Euclidean lines in  real hyperbolic geometry - geodesics and horocycles (cycles of infinite radius). In higher dimensions we correspondingly have two   substitutes for the Euclidean planes - the totally geodesic submanifolds and horospheres.
The term {\it horosphere} was introduced by  Lobachevsky who used the word  ``orisphere''. It means a sphere of infinite radius. The concept itself  seems to go far back to Gauss's student Friedrich Wachter. \footnote{See, e.g., http://en.wikipedia.org/wiki/Talk:Horosphere.}
 The horosphere  can be obtained if we
 take a geodesic sphere in $\hn$ and allow the center to move to infinity, still requiring the sphere to pass through some fixed point.   In the hyperboloid model, the horospheres
 are represented by intersections of the hyperboloid $\hn$ with hyperplanes whose normal lies in the  asymptotic
cone.

The {\it horospherical Radon transform}  $\frH f$  assigns to each sufficiently
 good function $f: \hn \to \bbc$ the integrals of $f$ over horospheres.  It is also called the  {\it Gelfand-Graev transform}; see   \cite[p. 290]{GGV},   \cite[p. 532]{V},  \cite[p. 162]{VK}. In these publications, a compactly supported smooth function $f$ is reconstructed from $\frH f$   in terms of divergent integrals that should be understood in the sense of distributions; see (\ref{Horo2}) below.
 On the other hand, for another widely known class of Radon-like transforms, namely, the totally geodesic transforms in constant curvature spaces,  inversion formulas are available (a)  in terms of the mean value operators, according to the general Funk-Radon-Helgason scheme, and (b)
  in terms of polynomials of the Beltrami-Laplace operator, which arise as left inverses of the corresponding potentials; see \cite {H11, Rou, Ru02a, Ru02c, Ru13b}. The method (a) is also applicable to $L^p$ functions.

 The aim of the present paper is to show that both methods (a) and (b) are well-suited  for the horospherical  transform. To this end, we develop the pertinent tools of fractional integration and harmonic analysis on $\hn$. In particular,
  we introduce a new analytic family of potential type operators, which serve as substitutes for Riesz potentials and sine transforms in the totally geodesic case; cf.  \cite {H11, Ru02b,  Ru02c}. These potentials can
 be inverted by polynomials of the Beltrami-Laplace operator on $\hn$. We also introduce the horospherical analogues of the Semenistyi type integrals \cite{Se1}. These integrals form a meromorphic operator family, including the horospherical transform and its dual. Modifications of such integrals for diverse Radon-like transforms have proved to be a powerful tool in  integral geometry and are of interest from the point of view of analysis; see \cite{Ru13c} and references therein.

The horospherical  transform also appears in the literature under the name  {\it  ``horocycle transform''} and can be treated in the general context of symmetric spaces. More information on this subject can be found in the works by Helgason  \cite{H73, H08, H11, H12}, Gindikin \cite{Gi01, Gi08, Gi13},  Gonzalez \cite{Go10a},  Gonzalez and Quinto \cite{GQ94}, Hilgert,  Pasquale, and  Vinberg  \cite{HPVa, HPV},
Zorich \cite{Zo1, Zo2}; see also   Berenstein and Casadio Tarabusi \cite{BC94},  Bray and Solmon \cite{BS90},  Bray and Rubin \cite{Bru},  Katsevich \cite{Kat05}.
 The methods and results of these publications essentially differ from those in the  present article.

{\bf Plan of the paper.} Section 2 contains geometric and analytic preliminaries. In Section 3 we study basic properties of the horospherical transform $\frH f$. Section 4 is devoted to inversion formulas for $\frH f$ on  functions $f\in C_c^\infty (\hn)$ and $f\in L^p (\hn)$. The main  results are stated in Theorems \ref{jOOOthERC2} and \ref{ThHORYP}.

\section{Preliminaries}

\subsection{Basic Definitions}\label {PPMMBRE}

The  pseudo-Euclidean space $\bbe^{n, 1}$, $n\ge 2$, is  the  $(n+1)$-dimensional real vector space of points
in $\bbr^{n +1}$ with the inner product
\be\label {tag 2.1-HYP}[{\bf x}, {\bf y}] = - x_1 y_1 - \ldots -x_n y_n + x_{n +1} y_{n +1}. \ee
 The  distance $\|{\bf x}-  {\bf y} \|$ between two points  in $\bbe^{n, 1}$ is defined by
\[ \| {\bf x}\!-\!  {\bf y} \|^2\! =\! [{\bf x}\!-\!  {\bf y}, {\bf x}\!-\!  {\bf y}\!]=\! - (x_1\! -\! y_1)^2 \!- \ldots - \!(x_n\! -\! y_n)^2\! +\! (x_{n +1}\! -\!
y_{n +1})^2, \]
so that   $\| {\bf x}-  {\bf y} \|^2$ can be positive, zero, and negative.  For the corresponding cones in $\bbe^{n, 1}$ we use the notation
\[ \Gam =\{{\bf x} \in \bbe^{n,1}: \, [{\bf x}, {\bf x}]=0\}, \qquad \Gam_{\pm}=\{{\bf x} \in \Gam: \, \pm\, x_{n+1} >0\}.\]
The pseudo-orthogonal group of linear transformations preserving the bilinear form $[{\bf x},  {\bf y}]$ is denoted by $O(n,1)$. The special pseudo-orthogonal group $SO(n,1)$ is the subgroup of $O(n,1)$ consisting of all elements with determinant $1$.  The group $SO(n,1)$ is not connected and has two connected components.  The notation
\[G=SO_0(n,1)\]
is used for the identity component of $SO(n,1)$. The elements of $G$ are called {\it hyperbolic rotations} or {\it pseudo-rotations}. The action of $G$ splits the space $\bbe^{n,1}$ into orbits of the following forms: 1) upper  sheets of two-sheeted hyperboloids, 2) lower  sheets of the same hyperboloids, 3) one-sheeted hyperboloids, 4) the upper sheet $\Gam_+$ of the cone $\Gam$, 5) the lower sheet $\Gam_-$ of the cone $\Gam$, 6) the origin $o=(0, \ldots, 0)$.

 Let  $K=SO(n)$ and $H=SO_0(n-1, 1)$  be the  subgroups of $G$, the elements of which fix the $x_{n+1}$-axis and the hyperplane $x_n=0$, respectively.
The $n$-dimensional real hyperbolic space $\hn$ is realized as the upper sheet of the two-sheeted hyperboloid   $\| {\bf x} \|^2 = 1$, that is,
\[\hn = \{{\bf x}\in \bbe^{n,1} :
\| {\bf x} \|^2 = 1, \ x_{n +1} > 0 \}.\]

The points of $\hn$  will be denoted by the non-boldfaced letters, unlike the generic points in $\bbe^{n,1}$.
The geodesic distance between the points $x$ and $y$ in $\hn$ is defined by $d(x,y) = \cosh^{-1}[x,y]$, so that
\[ [x,a]=\ch r\]
is the equation of the geodesic sphere in $\hn$ of radius $r$ with center at $a\in\hn$.

We denote by $ \ e_1, \ldots, e_{n +1}$  the coordinate unit vectors  in $\bbe^{n,1}$;  $S^{n -1}$ is  the unit sphere  in the coordinate plane $\bbr^n=\{x\in \bbe^{n,1} : x_{n +1}=0\} $; $\sigma_{n-1} =  2\pi^{n/2} \big/ \Gamma (n/2)$ is the
surface area of $S^{n-1}$. For    $\theta \in S^{n-1}$, $d\theta$ denotes the surface element on $S^{n-1}$;
$d_*\theta= d\theta/\sigma_{n-1}$ is the normalized surface element on $S^{n-1}$.
The point $x_0=(0, \ldots, 0,1)\sim e_{n +1}$ serves as the origin of $\hn$; $\bbh^{n-1}=\{x\in \hn : x_n=0\}$.

%The groups  $K=SO(n)$ and $H=SO_0(n-1, 1)$  are the isotropy subgroups of $x_0$ and $\bbh^{n-1}$, so that $\hn$ and $\hns$ are isomorphic to %the quotient spaces $G/K$ and $G/H$, respectively. We write $\hn=G/K$ and $\hns=G/H$. The $K$-invariant  functions on $\hn$ and $\hns$ are %called zonal or radial.

The hyperbolic coordinates of a point
$x = (x_1, \ldots, x_{n +1}) \in \hn$ are defined by
\be\label {tag 2.3-HYP} \left \{\begin{array} {l} x_1 = \sh r \,\sin \theta_{n -1} \ldots
\sin \theta_2 \,\sin \theta_1, \\
x_2 = \sh r \,\sin \theta_{n -1} \ldots \sin \theta_2 \,\cos \,\theta_1, \\
 .....................................\\
x_n = \sh r \,\cos \,\theta_{n -1}, \\
x_{n +1} = \ch r, \end{array} \right. \ee
where $0 \le \theta_1 < 2 \pi; \ 0 \le \theta_j < \pi, \ 1 < j \le n - 1; \ 0 \le r < \infty$. In other words,
\be\label {taddd-HYP}  x = \theta\, \sh r  + e_{n+1} \, \ch r,\ee
where $\th$ is a point in $S^{n -1}$ with spherical coordinates  $\theta_1, \ldots, \theta_{n -1}$.

It is important to take special care of the consistency of all invariant measures in our consideration.
We fix a $G$-invariant measure $dx$ on $\hn$, which has the following form in the coordinates (\ref{taddd-HYP}):
\be\label {kUUUPqs}  d x = \sh^{n -1} r \, d r d \theta.\ee
 Then the Haar measure $dg$ on $G$ will be normalized in a consistent way by the formula
 \be\label {tag 2.3-AIM}
\intl_G f(ge_{n+1})\,dg=\intl_{\hn} f(x)\,dx.\ee
If $f$ is $K$-invariant, that is,    $f(x)\equiv f_0 (x_{n+1})$, then
\be\label{ppooii}
\intl_{\hn} f(x)\, dx=\sig_{n-1}\intl_1^\infty f_0(s) (s^2 -1)^{n/2 -1}\, ds.\ee

 The notation  $ u\cdot v  =u_1 v_1 + \ldots + u_n v_n$  is used for the
usual inner product of the vectors $u, v \in \bbr^n$;
$C(\hn)$ is the
space of continuous functions on $\hn$; $C_0 (\hn)$ denotes the space of continuous functions
on $\hn$ vanishing at infinity.
We also set
\be\label{benatur} C_\mu (\hn)=\{f\in C(\hn): f(x)= O(x_{n+1}^{-\mu})\}.\ee
Let $\Om=\{{\bf x}\in E^{n,1}: ||{\bf x}||^2>0, \,x_{n+1} >0\}$ be
 the interior of  the cone $\Gam_+$; $ B=\{\xi \in \Gam_+ \colon \, \xi_{n+1}=1\}$. We denote by
  $ C^\infty_c(\hn)$ the space of  infinitely differentiable compactly supported  functions on $\hn$. This space is formed by the restrictions onto $\hn$ of functions belonging to $C^\infty_c (\Om)$.

We say that  an integral under consideration
 exists in the Lebesgue sense if it is finite when the integrand is replaced by its absolute value.

\subsection{Spherical Means and Hyperbolic Convolutions}\label{pUUdu}

 Given $x \in \bbh^n$ and $ s > 1$, let
 \be\label {2.21hDIF}
 (M_x f)(s) = {(s^2-1)^{(1-n)/2}\over \sigma_{n-1}}
 \intl_{\{y \in \bbh^n:\; [x, y]=s\}} f(y)\, d\sigma  (y),\ee
 where $ d\sigma  (y)$
 stands for the relevant induced Lebesgue measure. The integral (\ref{2.21hDIF}) is the mean value of $f$ over the planar section
$\Gamma_x (s) = \{ y \in \hn:
[x, y] = s\}$.  This section is a geodesic sphere of radius $r=\ch^{-1} s$ with center
at $x$, so
that \[\intl_{\Gamma_x (s)} d\sigma  (y) = \sigma_{n -1} \,(s^2 -1)^{(n-1)/2}.\]

It is clear that if $f$ is compactly supported in $\hn$, then $(M_x f)(\cdot)$  is compactly supported in $[1,\infty)$ for every $x$.

 Let $\om_x \in G$ be a hyperbolic rotation which takes $e_{n+1}$ to $x$. Changing variables and setting $f_x (y)=f (\om_x y)$, we have
  \bea\label {AAAhDIF9}
 (M_x f)(s) &=& \intl_{\sn}f_x(\th \sqrt{s^2-1}+ e_{n+1}s)\, d_*\th\\
&=&\label {AAAhDIF} \intl_{\sn}f_x(\th \,\sh r+ e_{n+1}\, \ch r)\, d_*\th, \eea
 or
  \be\label {2.21hDIFy}
 (M_x f)(\ch r)= \intl_K f_x(k a_r e_{n+1})\, dk, \ee
 \[ a_r=\left[\begin{array} {ccc} I_{n-1} &0 &0\\
0 &\cosh r &\sinh r\\ 0 &\sinh r &\cosh r\end{array}\right ]. \]

The following statement is  due to Lizorkin; cf. \cite[pp. 131-133]{Liz93}. We presented it in a slightly different form.

\begin{lemma}\label {Lemma 2.1-HYP}  Let $f \!\in \!L^p (\hn)$, $ 1\! \le \!p \!\le \!\infty$. Then
\be\label {2.21hDIFb}  \sup\limits_{s > 1} \| (M_{(\cdot)} f)(s) \|_p \le
\| f \|_p.\ee
  If $1 \le p < \infty$, then  $(M_x f)(s)$ is a continuous $L^p$-valued function of  $s\in [1,\infty)$ and
\be\label {2.XIFb} \lim\limits_{s \to 1} \| (M_{(\cdot)} f)(s) - f \|_p = 0.\ee  If $f \in C_0 (\hn)$, then $(M_x f)(s)$ is a continuous function of $(x,s)\in \hn \times (1,\infty)$ and
$(M_x f)(s)\to f(x)$
as $s\to 1$, uniformly on $\hn$.
 \end{lemma}
 \begin{proof}   By the generalized Minkowski inequality, owing to  (\ref{tag 2.3-AIM}), we have
\bea \| M_r f \|_p  &=& \Bigg  (\int\limits_G | (M_{g e_{n +1}} f)
(s) |^p \,d g  \Bigg  )^{1/p} \nonumber\\
&=& \frac{(s^2-1)^{(1-n)/2}}{\sigma_{n -1}} \,  \Bigg ( \int\limits_G \, \Big |\,
\int\limits_{z_{n +1} = s}
f ( g z) \,  d\sigma  (z) \Big |^p \,d g \Bigg )^{1/ p} \nonumber\\
&\le&  \frac{(s^2-1)^{(1-n)/2}}{\sigma_{n -1}}\, \|
f \|_p \int\limits_{z_{n +1} =s}
d\sigma  (z) = \| f \|_p. \nonumber\eea

Let us prove the second statement. By (\ref{2.21hDIFb}),
it suffices to consider the case when $f$ belongs to the dense subset $C_c^\infty (\hn)$.
  Suppose $s_1, s_2 \in [1, \infty)$, $s_1=\ch\, r_1$, $s_2=\ch\, r_2$. By (\ref{2.21hDIFy}),
\bea
I&\equiv&\| (M_{(\cdot)} f)(s_1) -  (M_{(\cdot)} f)(s_2) \|_{L^p (\hn)} \nonumber\\
&\le&\intl_{K}\| f (g ka_{r_1} e_{n +1})-f (g ka_{r_2} e_{n +1})\|_{L^p (G)}\, dk\nonumber\\
&=&\| f (g a_{r_1} e_{n +1})-f (g a_{r_2} e_{n +1})\|_{L^p (G)}.\nonumber\eea
Hence,
\[ I^p \le \intl_G |  f (g a_{r_2}^{-1}a_{r_1}e_{n +1})-  f (g e_{n +1})|^p\, dg\to 0 \quad \text{\rm as $\;|r_1-r_2| \to 0;$}\]
 see, e.g.,  Hewitt and Ross \cite [Ch. 5, Sec. 20.4] {HR}.
The case $s_2\!=\!1$ gives  (\ref{2.XIFb}).

For $f \in C_0 (\hn)$,  the continuity of the function $(x,s) \to (M_x f)(s)$ on $\hn \times (1,\infty)$ follows from the definition of $(M_x f)(s)$. The proof of the limit formula is similar to that in the $L^p$ case with $\| \cdot \|_p$ replaced by the pertinent sup-norm.
\end{proof}

 Given a measurable function $k$ on $[1, \infty)$, the
corresponding  hyperbolic convolution on $\hn$ is defined by
\be\label {tag 2.16-HYP} (K f) (x) = \int\limits_{\hn} k ([x, y]) f (y)  \, d y, \qquad x \in \hn, \ee
provided that the integral on the right-hand side is finite.

\begin{lemma}\label {Lemma 2.2-HYP}   Let $1 \le p \le \infty$. If
\[c = \sigma_{n -1} \intl^\infty_1 |k (s)| \, (s^2 - 1)^{n/2 -1} \,d s < \infty,\]  then  $(K f) (x)$
 is finite for almost all $x$, and
$\| K f \|_p \le c \,\| f \|_p$.
\end{lemma}
\begin{proof}  Let $\omega_x \in SO_0 (n, 1)$ be a pseudo-rotation that takes $x_0= e_{n +1}$ to $x$.
We put  $y = \omega_x z$ to get $[x,y] = z_{n +1}=\ch r$. Then, by Fubini's theorem,
\be\label {tag 2.17-HYP} \int\limits_{\hn} k ([ x, y]) f (y) \! \; d y =
\sigma_{n -1} \int\limits^\infty_0 k (\ch r) \ (M_x f) (\ch r) \
\sh^{n -1} r \ d r, \ee
where $M_x f$ is the spherical mean (\ref{2.21hDIF}).  Owing to (\ref{2.21hDIFb}), the result
follows by the generalized Minkowski inequality. \end{proof}

The integral (\ref{tag 2.16-HYP})  can be ``lifted" to a convolution operator on $G$.
By Young's inequality (see, e.g., Hewitt and Ross  \cite[Chapter 5,
 Theorem 20.18]{HR}), we have
\be \label{aaqqwwz} \| Kf\|_q\le \| f\|_p \|k\|_r \, ,\ee
 where $1 \le p \le q \le \infty, \quad 1 - p^{-1}+q^{-1}=r^{-1}$,
\[
 \| k\|^r_r=\sigma_{n-1}\intl^\infty_1|k(t)|^r\; \; (t^2-1)^{n/2-1}dt.
\]

\subsection{Selected Aspects of the Fourier Analysis on $\bbh^n$}

This area is very large. More details  and further references can be found in Bray \cite{Bray94},  Gelfand,  Graev, and
Vilenkin \cite{GGV}, Faraut \cite{Far79, Far82, Far83},    Flensted-Jensen  and   Koornwinder \cite{FJK73},   Helgason \cite{H00},   Molchanov  \cite{Mol66, Mol76},  Rossmann \cite{Ross}.  Many related results in the literature are presented in the general context of Riemannian symmetric spaces. Below we  briefly review some basic facts.

As before,  $\Om=\{{\bf x}\in E^{n,1}: ||{\bf x}||^2>0, \,x_{n+1} >0\}$ denotes
 the interior of  the cone $\Gam_+ = \{ {\bf x}\in E^{n,1}: \, ||{\bf x}|| = 0, \; \; x_{n+1}>0\}$,   $ B=\{\xi \in \Gam_+ \colon \, \xi_{n+1}=1\}$.

  For  $x\in \hn$ and   $\xi\in \Gam_+$, we have
$[x, \xi]>0$. Indeed, assuming the contrary, for $x=(x_1,\ldots, x_n, x_{n+1})=(x',x_{n+1})$ and $\xi=(\xi_1,\ldots, \xi_n, \xi_{n+1})=(\xi',\xi_{n+1})$, by Schwarz's inequality we have
\[x_{n+1}\xi_{n+1}\le x' \cdot \xi'\le |x'|\, |\xi'|=|x'|\,\xi_{n+1},\]
because $[\xi,  \xi]=-|\xi'|^2+\xi_{n+1}^2=0$. Hence, $x_{n+1}\le |x'|$ that contradicts $[x,x]=1$.

 The Beltrami-Laplace operator $\Delta_H $ on $\bbh^n$ is the tangential part of the d'Alembertian
\[
\square = \frac{\partial^2}{\partial x^2_1} + \ldots +\frac{\partial^2}{\partial x^2_n}-\frac{\partial^2}{\partial x^2_{n+1}}
\]
on $\Om$. Namely, if
\[
{\bf x}= t(\theta\, \sh r  + e_{n+1} \, \ch r) \in\Om,  \qquad t=||{\bf x}||, \quad r \ge 0, \quad \th\in S^{n-1},\] then
\[
\square=-\Big(\frac{\partial^2}{\partial t^2}+ \frac{n}{t}\frac{\partial}{\partial t}\Big) + \frac{1}{t^2} \Delta_H, \qquad \Delta_H=\Del_r + \frac{1}{\sinh^2 r}\,\Delta_S,
\]
\be\label {tPPPay}  \Del_r=\frac{\partial^2}{\partial r^2} + (n-1)\,\coth r\frac{\partial}{\partial r},\ee
$\Delta_S$ being the Beltrami-Laplace operator on $S^{n-1}$.

 For $\om\in S^{n-1}\! \subset \!\bbr^n$, we set $b(\om) \!=\! (\om, 1)\! \in \!B$.   Since   the function $g_\om({\bf x})=[{\bf x}, b(\om)]^\mu=t^\mu [x, b(\om)]^\mu$, with $\mu \in \bbc$, $t>0$, and $x\in \hn$, satisfies  $\square g =0$ in $\Om$, then $x \to [x, b(\om)]^\mu$ is an eigenfunction of $\Delta_H$. The  corresponding eigenvalue is computed straightforward, using the radial part of $\square $, as follows.

\begin{lemma}\label{Lemma 3.1Bray}    Let $\mu \in \bbc, \; \; \om\in S^{n-1}$.  Then $x \to [x, b(\om)]^\mu$ is an eigenfunction of $\Delta_H$ with the eigenvalue $\mu(\mu-1+n)$.
In particular,
\be\label {taTTRray}
\Delta_H [x, b(\om)]^{i\lambda - \del}=
-  (\lambda^2 +\del^2)\,[x, b(\om)]^{i\lambda - \del}, \qquad \lam \in  \bbr,
\ee
where
\[\del = (n-1)/2.\]
\end{lemma}

\begin{corollary}\label{Corollary 3.2}  The  function
\be\label {tag 3.1Bray}
\Phi_\lambda(x) =\intl_{S^{n-1}}[x, b(\om)]^{i\lambda-\del}\,d_*\om,   \qquad x\in \hn,\ee
is an eigenfunction of $\Delta_H$ with the eigenvalue $-(\lambda^2+\del^2)$.
\end{corollary}

The integral    (\ref{tag 3.1Bray}) is called the {\it spherical function} of $\Delta_H$.
It is well-defined, because $[x, b(\om)]>0$.
If $x=\th\,\sh r +e_{n+1}\, \cosh r $, $ \; r \ge 0$, $ \th \in S^{n-1}$, then
\[ \Phi_\lambda (x)=\frac{\sigma_{n-2}}{\sigma_{n-1}}\,\intl^\pi_0(\cosh r \,- \sinh r \, \cos\, \psi)^{i\lambda-\del} (\sin\psi)^{2\del-1}d\psi,\]
so that $\Phi_\lambda$ is zonal. We write $\Phi_\lambda(x) =\tilde\Phi_\lambda (r)$. Then, by Erd\'elyi \cite[3.7(7)]{Er},
\be
\label {tag 3.2Bray} \tilde\Phi_\lambda (r)=2^{\del-1/2}\Gamma(\del+1/2)(\sh r)^{1/2-\del}
P^{1/2-\del}_{i\lambda-1/2} (\ch r),\ee
$P^\mu_\nu(z)$ being the associated Legendre function.

The Fourier transform of a function $f\in C^\infty_c(\hn)$ is defined by
\be\label {tag 3.3Bray}
\tilde f (\lambda, \om)=\intl_{\hn} f(x)\, [x, b(\om)]^{i\lambda-\del}\,dx,\qquad \lambda\in\bbr, \; \; \om\in S^{n-1}.
\ee
For  $f\in C^\infty_c(\hn)$, the following formula holds (cf. (\ref{taTTRray})):
 \be\label{tag 4.12OFR1}
(\Delta_H f)^\sim (\lambda, \om)= -(\lambda^2+\del^2)\tilde f(\lambda,\om).\ee
  If $f$ is zonal,  that is, $f(x)\!=\!f_0(x_{n+1})$, then $\tilde f (\lambda, \om)\equiv\tilde f (\lambda)$ is independent of $\om$ and
\be\label {tag 3.4Bray}
\tilde f (\lambda)\!=\!\sigma_{n-1} \intl^\infty_0 \!f_0(\cosh r)\,\tilde\Phi_\lambda (r)\,(\sinh r)^{n-1} dr \!=\! \intl_{\hn} \!\! f(x)\Phi_\lambda(x)\,dx.\ee
This expression is called the {\it  spherical transform} of the zonal function $f$.  More general Fourier-Jacobi transforms and related convolution operators were studied by  Flensted-Jensen  and  Koornwinder \cite{FJK73}.

\begin{lemma}\label{Lemma 3.3Bray} Let $\Phi_0(x) = \Phi_\lambda(x)|_{\lambda=0}$,
\be\label {tag 3.5Bray} (Kf)(x)=\intl_{\hn}  f(y)k([x,y])\, dy = (f*k)(x), \qquad x\in \hn.\ee
Suppose that
\[\hbox{\rm (a)}\intl_{\hn} |f(x)|\,\Phi_0(x)\,dx<\infty, \qquad \qquad
 \hbox{\rm (b)}
 \intl_{\hn} |k(x_{n+1})|\Phi_0(x)\,dx<\infty.
\]
Then, for all $\lambda \in\bbr$ and almost all $\om\in S^{n-1}$, the Fourier transforms $\tilde k(\lambda), \; \tilde f(\lambda, \om)$, and $ (Kf)^\sim (\lambda, \om)$ are finite. Furthermore,
\be\label {tag 3.6Bray}
(Kf)^\sim (\lambda,\om)=\tilde k(\lambda)\tilde f(\lambda,\om).\ee
\end{lemma}
\begin{proof}   Changing the order of integration (this step will be justified later),  we get
\be\label {tag 3.7Bray}
(Kf)^\sim(\lambda,\om)\buildrel{(!)}\over =\intl_{\hn} f(y)\,dy\intl_{\hn} [x, b(\om)]^{i\lambda-\del}k([x, y])\,dx.\ee
Let $x=r_yz$, so that  $ r_y\in G$, $r_y e_{n+1}=y$, and let $ r^{-1}_yb(\om)=\xi\in\Gam$. By the homogeneity, $\xi=\xi_{n+1}b(\tilde\om)$ for some $\tilde\om\in S^{n-1}$. Then the  inner integral can be written as
\[
\intl_{\hn} k(z_{n+1})[z, \xi]^{i\lambda-\del}\,dz=\xi_{n+1}^{i\lambda-\del}\intl_{\hn}  k(z_{n+1})[z, b(\tilde\om)]^{i\lambda-\del}\,dz=\xi^{i\lambda-\del}_{n+1}\tilde k(\lambda).\]
 Since $\xi_{n+1} = [\xi, e_{n+1}]=[y, b(\om)]$, the result follows.  To justify (!) in (\ref{tag 3.7Bray}), it suffices to note that
 for nonnegative $f$ and $k$ we have $(Kf)^\sim(0, \om) = \tilde f (0, \om)\tilde k(0)$.  This expression is finite for almost all $\om$, because $\hat k(0)<\infty$ (by (b)) and
\[\intl_{\sn} \tilde f (0,\om)\, d\om= \sigma_{n-1}\intl_{\hn} f(y)\Phi_0(y)\,dy<\infty\]
(by (a)).  Thus, the repeated
integral in (\ref{tag 3.7Bray}) is absolutely convergent and the change of the order of integration is justified.
\end{proof}

The following statement is a particular case of Theorem 3.2 from Flensted-Jensen  and  Koornwinder  \cite{FJK73}.
\begin{theorem}\label{Theorem 3.5Bray}  If $1\le p \le 2$, then the spherical transform (\ref{tag 3.4Bray}) is injective on the space $L^p_{z}(\hn)$ of zonal functions in $L^p(\hn)$.
\end{theorem}

\begin{example}\label{Example 3.4Bray} {\rm Let $n\ge 2$, $\;x=\th\,\sh r + e_{n+1}\ch r$,  $\;r=d(e_{n+1}, x)$. We set
\[
q_\a (x)\!=\!\zeta_{n,\a} \,\frac{(\ch r \!-\!1)^{(\a-n)/2}}{(\ch r \! +\!1)^{n/2-1}}, \qquad  \zeta_{n,\a}= \frac{\Gamma((n-\a)/2)}{2^{\a/2 +1}\,\pi^{n/2}\Gamma(\a/2)}.\]
The spherical Fourier transform of $ q_\a$,  can be explicitly evaluated. Specifically, for all $\lambda \in\bbr$ and $0 < Re\,\a<n-1$,
\be\label {tag 3.8BrayB}
\tilde q_\a (\lambda)= \frac{\displaystyle{\Gamma \Big(\frac{n- 1}{2} - \frac{\a}{2} +i\lambda\Big)\, \Gamma\Big(\frac{n- 1}{2}-\frac{\a}{2}-i\lambda\Big)}}{\displaystyle{ \Gamma \Big(\frac{n - 1}{2}+ i\lambda\Big)\,\Gamma \Big(\frac{n- 1}{2}-i\lambda\Big)}}.  \ee
This equality    follows from (\ref{tag 3.4Bray}) and (\ref{tag 3.2Bray}), owing to the formula 2.17.3(6) from  \cite {PBM3}.
The convolution operator
\be\label {tMUXXXXX}
Q^\a f=q_\a *f\ee will play an important role in our consideration.
}
\end{example}

\subsection{The Operator $Q^\a$} \label{ExaSinreQ}

According to  Example \ref  {Example 3.4Bray},
  for $Re\,\a >0$, $\a-n\neq 0, 2, 4, \dots$, we have
\be\label {BGBBQQ}
(Q^\a f)(x)\!=\!\zeta_{n,\a}\intl_{\hn}\!\! f(y)\, \frac{([x,y]-1)^{(\a -n)/2}}{([x,y]+1)^{n/2 -1}}\, dy. \ee
This operator will serve as an analogue of the Riesz potential in the inversion procedure for the horospherical transform in Section 4; see Lemmas \ref{duJJYPhor} and \ref{MMM-HYP3horT}.

The following statement  holds by  Young's inequality (\ref{aaqqwwz}).

\begin{proposition}  \label {tag LLyHOR} Let $f \in L^p(\hn), \; \; 1\le p \le \infty, \; \; 0 < \a < 2(n-1)/p$. Then $(Q^\a f)(x)$ exists
as an absolutely convergent integral (a) for almost all $x$ if $0 < \a
\le n/p$, and (b) for all $x$ if $\a > n/p$.  If \be
\frac{1}{p} - \frac{\a}{n} < \frac{1}{q} < \frac{1}{p} -
\frac{\a}{2(n-1)}, \ee then $\| Q^\a f\|_q\le c \,\| f\|_p,\quad c =
c(\a, n, p)$.  In particular,
\be Q^\a\colon L^p(\hn)\to
L^\infty (\hn)\;\; \hbox{\rm if } \; \; n/p\!<\!\a\! <\!2(n-1)/p,\quad 1\! \le\! p\!
< \!\infty. \ee
\end{proposition}

 \begin{lemma}\label {LemTheor 4.4HYQQ} Let $f\in C^\infty_c(\hn)$, $ \a \ge 2$, $ D_\a = - \Delta_H-\a (2n-2-\a)/4$. If $\a-n\neq 0, 2, 4, \ldots$, then
\be\label{tag 4.8OFR1QQ}
D_\a Q^\a f = Q^{\a-2} f \qquad (Q^0 f=f).\ee
\end{lemma}
{\it Sketch of the proof.} A formal application of the
 Fourier transform gives \[\tilde D_\a (\lam)\, \tilde q_\a (\lam)= \tilde q_{\a-2} (\lam)\] and therefore, by (\ref{tag 3.8BrayB}),
 \bea
 \tilde D_\a (\lam)&=&\frac{\tilde q_{\a-2} (\lam)}{\tilde q_\a (\lam)}= \Big(\frac{n- 1-\a}{2}  +i\lambda \Big)
 \Big(\frac{n- 1-\a}{2}-i\lambda\Big)\nonumber\\
 &=&\del^2  +\lam^2-\del \a  + \frac{\a^2}{4}, \qquad \del = (n-1)/2.\nonumber\eea
 Since  $(\Delta_H f)^\sim (\lambda, \om)= -(\lambda^2+\del^2)\tilde f(\lambda,\om)$,
(\ref{tag 4.8OFR1QQ}) follows.
 A rigorous  proof relies on the Darboux-type equation
\be\label{tag 4AAOFR1}
\Delta_x \, [(M_{x} f)(\ch r)]  = \Delta_{r}\, [(M_{x}f)(\ch r)]\ee
and the subsequent integration by parts; cf. the  proof  of Theorem 1.9 in Helgason \cite[p. 125]{H11}.
Here $(M_{x} f)(\cdot)$ is the spherical mean (\ref{2.21hDIF}),
$\Delta_x$ stands for the Beltrami-Laplace operator $\Delta_H$ acting in the $x$-variable, and $\Delta_{r}$
 is the  radial part of $\Delta_H$; cf. (\ref{tPPPay}).  The equality (\ref{tag 4AAOFR1}) is transparent in the  Fourier terms, because
\be\label{tag 4.13OFR1}
[(M_{(\cdot)} f)(\ch r)]^\sim (\lambda, \om)= \tilde\Phi_\lambda (r)\tilde f (\lambda,\om),  \ee
\be\label{tBBOFR1}  \Delta_{r}\tilde\Phi_\lambda (r) = - (\lambda^2+\del^2)\tilde \Phi_\lambda(r),\ee
 $\tilde\Phi_\lambda (r)$ being the function (\ref{tag 3.2Bray});
  see, e.g.,  Petrova \cite[Section 4]{Pet93}.

 $\hfill \square$

Lemma \ref {LemTheor 4.4HYQQ} implies the following
\begin{proposition}\label{Corollary 4.5HY}  Let $f\in C^\infty_c(\hn)$, $\P_\ell(\Delta_H) =  D_{2}D_{4}\ldots D_{2\ell}, \quad \ell \in \bbn$,  where $D_\a = - \Delta_H-\a (2n-2-\a)/4$.    If  $2\ell-n\neq 0, 2, 4, \ldots $, then
\be\label{tYUYUFR1} \P_\ell(\Delta_H)Q^{2\ell}f=f.\ee
\end{proposition}

For further purposes, we need an extension of Lemma \ref {LemTheor 4.4HYQQ} to the case $\a=n$. If $f\in C^\infty_c(\hn)$, we define  $Q^n f$ as a limit
\bea (Q^n f)(x)\!\!&=&\!\!\lim\limits_{\a\to n }\zeta_{n,\a}\intl_{\hn}\!\! f(y)\, \frac{([x,y]-1)^{(\a -n)/2}-1}{([x,y]+1)^{n/2 -1}}\, dy\qquad\nonumber\\
\label{tZ44FR1} \!\!&=&\!\! \zeta'_{n}\intl_{\hn}\!\! f(y)\, \frac{\log ([x,y]\!-\!1)}{([x,y]\!+\!1)^{n/2 -1}}\, dy, \quad \zeta'_{n}=-\frac{2^{-1-n/2}}{\pi^{n/2}\, \Gam (n/2)}.\quad \qquad \eea

\begin{lemma}\label {Le76768QQ} Let $f\in C^\infty_c(\hn)$, $ D_n = - \Delta_H -n (n-2)/4$, $n\ge 2$. Then
\be\label{tYUY34}  D_n Q^n f= Q^{n-2} f +Bf\qquad (Q^{0} f=f),\ee
where
\be\label{tYUY341} (Bf)(x)=\zeta'_{n}\intl_{\hn} f(y)\frac{dy}{([x,y]\!+\!1)^{n/2 -1}}.\ee
\end{lemma}
\begin{proof} Using (\ref{tag 2.17-HYP}) and the Darboux-type equation (\ref{tag 4AAOFR1}), we can write
\[
- (\Delta_H Q^n f)(x)=-\sigma_{n -1}\zeta'_{n} \int\limits^\infty_0 a (r) \ (g_x'' (r)+(n-1) \, {\rm coth} r\, g_x'(r))  \, d r,\]
where  $g_x(r)=(M_x f) (\ch r)$, $\;g_x(0)=f(x)$,
\[   a (r)=\frac{\log (\ch r -1)}{(\ch r +1)^{n/2-1}}\, (\sh r)^{n-1}.\]
The rest of the proof is a routine integration by parts.
\end{proof}

Our next goal is to apply Lemma \ref {LemTheor 4.4HYQQ} to (\ref{tYUY34}) and reduce the order of the potential $Q^{n-2} f$.

\begin{lemma}\label {Le76768QQ1}  Let $f\in C^\infty_c(\hn)$,  $n> 2$. Then $(Bf)(x)$ is an eigenfunction of the Beltrami-Laplace operator $\Del_H$, so that
\be\label{tYUY342} - \Del_H Bf=\frac{n(n-2)}{4}\, Bf\ee
and
\be\label{tRRRYUY341}
D_n Bf=D_{n-2} Bf=0.\ee
\end{lemma}
\begin{proof} As in the proof of Lemma \ref{Le76768QQ}, we integrate by parts. Let
\[b(r)=\frac{(\sh r)^{n-1}}{ (\ch r +1)^{n/2-1}}, \qquad g_x(r)=(M_x f) (\ch r).\] By (\ref{tag 2.17-HYP}) and  (\ref{tag 4AAOFR1}),
\bea
- (\Delta_H B f)(x)\!\!\!&=&\!\!\!-\sigma_{n -1}\zeta'_{n} \int\limits^\infty_0 b (r) \ (g_x'' (r)+(n-1) \, {\rm coth} r\, g_x'(r))  \, d r\nonumber\\
\!\!\!&=&\!\!\!\sigma_{n -1}\zeta'_{n}\int\limits^\infty_0 g_x'(r) \, [b' (r)-(n-1) \, {\rm coth} r\, b(r)]\, dr\nonumber\\
\!\!\!&=&\!\!\!\sigma_{n -1}\zeta'_{n}\, \frac{n(n-2)}{4}\int\limits^\infty_0 g_x(r)\, \frac{(\sh r)^{n-1}}{(\ch r \!+\!1)^{n/2-1}}\, dr\nonumber\\
&=&\frac{n(n-2)}{4}\, (Bf)(x).\nonumber\eea
Further, since $D_n = - \Delta_H -n (n-2)/4$ and $D_{n-2}= - \Delta_H -(n-2)n/4$, then,  by (\ref{tYUY342}),
\[D_n Bf=D_{n-2} Bf=- \Delta_H Bf-[n (n-2)/4]\, Bf=0.\]
\end{proof}

Lemmas \ref {LemTheor 4.4HYQQ},  \ref{Le76768QQ} and \ref{Le76768QQ1} give the following statement.
\begin{proposition}\label{CoYYY5HY}  Let $f\in C^\infty_c(\hn)$, where $n$ is even.
If $n=2$, then
\be\label{tYUY343X}  - \Delta_H Q^2 f=f-\frac{1}{4\, \pi}\intl_{\bbh^2} f(y)\, dy.\ee
If $n\ge 4$, then
\be\label{tYUY343} \P_{n/2}(\Delta_H)\,Q^n f=f, \qquad \P_{n/2}(\Delta_H) = (-1)^{n/2} \prod\limits_{i=1}^{n/2}   (\Delta_H+i (n-1-i)).  \ee
\end{proposition}
\begin{proof} For $n=2$, the desired statement is contained in (\ref{tYUY34}). In the case $n\ge 4$ we
 need the notation  from Lemma \ref {LemTheor 4.4HYQQ}:
\be\label{tYUY343AB} D_\a = - \Delta_H-\a (2n-2-\a)/4, \qquad \a=2,4, 6, \ldots\, .\ee
Then  (\ref{tYUY34}) and (\ref{tag 4.8OFR1QQ}) yield
\[
\P_{n/2}(\Delta_H)Q^n f\equiv D_{2}D_{4}\ldots D_{n}Q^n f=f+D_{2}D_{4}\ldots D_{n-2}\,Bf.\]
Since, by (\ref{tRRRYUY341}), $D_{n-2}\,Bf=0$, the result follows.
\end{proof}

\section{The Horospherical  Radon Transform}\label {Horocycle}

\subsection{Preliminaries}

The main references for the following  prerequisites  are   Vilenkin and Klimyk  \cite{VK}, Gelfand,  Graev,  and
Vilenkin \cite{GGV},  Bray \cite{Bray94}.
As before,  $G=SO_{0}(n,1)$, $n\ge 2$,
${\bf x}=(x_{1},\ldots,x_{n+1}) \in \bbe^{n,1}$,  $x_0=(0, \ldots, 0,1)\sim e_{n+1}\in \hn$ (the origin of $\hn$);
\[ \Gam_+ =\{{\bf x} \in \bbe^{n,1}: \, [{\bf x}, {\bf x}]=0, \quad x_{n+1}>0\}, \qquad \xi_{0}=(0,\ldots,0,1,1)\in \Gam_+;\]
\[\rn=\{{\bf x}\in \bbe^{n,1}:\, x_{n+1}=0\}, \qquad \bbr^{n-1}=\{{\bf x}\in \bbe^{n,1}:\, x_{n}=x_{n+1}=0\}.\]
 The corresponding
rotation subgroups of $G$ are $K\!=\!SO(n)$ and $M\!=\!SO(n\!-\!1)$; $\sn\!=\!K/M$ is the unit sphere in $\rn$. The stabilizer of $\xi_{0}$ in $G$ consists of transformations of the form $g=g_1g_2$, $g_1\in M$,  $g_2\in N$, where
 the  subgroup $N$ is defined by
\[
N\!=\!\left\{n_{v}\!=\!\left[
\begin{array}
[c]{ccc}%
I_{n-1} & -v^{T} & v^{T}\\
v & 1\!-\!|v|^{2}/2 & |v|^{2}/2\\
v & -|v|^{2}/2 & 1\!+\!|v|^{2}/2
\end{array}
\right]  \,:\,v\in\bbr^{n-1}(\text{\rm the row vector})\right\}.
\]
Thus, since $G$ is transitive on $\Gam_+$, we can identify
\[\Gam_+=G/MN.\]
The Haar measure $dn_{v}$ on $N$ is given by the
Lebesgue measure $dv$ on $\bbr^{n-1}$, so that \[\intl_N f(n_{v}) \,dn_{v}=
\intl_{\bbr^{n-1}} f(n_{v})\, dv.\]

 \subsubsection{Horospherical Coordinates}

 Let $A$ be the
Abelian subgroup  of $G$ having the form
\[
A=\left\{ a_{t}=\left[
\begin{array}
[c]{ccc}%
I_{n-1} & 0 & 0\\
0 & \cosh t & \sinh t\\
0 & \sinh t & \cosh t
\end{array}
\right]  \, : \,t\in\bbr \right\}.
\]
One can readily see that $A$ normalizes $N,$  that is,
\be\label {normalizes} a_{t}^{-1}n_{v}a_{t}=n_{e^{-t}v}.\ee
 Every $x\in\hn$ can be  uniquely represented   as
\begin{equation}\label{horo coord}
x=n_{v}a_{t}x_0=a_{t}n_{e^{-t}v}x_0
 =(e^{-t}v,\,\sinh t+\frac{|v|^{2}}{2}e^{-t},\,\cosh t+\frac{|v|^{2}}{2}e^{-t}).
\end{equation}
We call $(v,t)$  the {\it horospherical coordinates} of $x$. Since $K=SO(n)$ is the stabilizer of $x_{0}$ in $G$, then (\ref{horo coord}) yields $G=NAK$,  the Iwasawa decomposition of $G$.

\begin{lemma}\label {Riemannian}
In   the horospherical coordinates, the invariant Riemannian
measure (\ref{kUUUPqs}) on $\hn$ has the  form
\be\label {kUUUPqs1} dx=e^{(1-n) t}\,dtdv.\ee
\end{lemma}
\begin{proof} It suffices to show that for every $f \in C_c (\hn)$,
\[ \intl_{\hn} f(x)\, d x=\intl_{\bbr}e^{(1-n) t} dt\intl_{\bbr^{n-1}} f(n_{v}a_{t}x_0)\, dv.\]
We write the left-hand side as
\[ I_l\!=\!\intl_{\rn} \frac{f (x', \sqrt{1 + |x'|^2})}{ \sqrt{1 + |x'|^2}}\, d x'\!=\!\intl_{\bbr} dx_n\!\intl_{\bbr^{n-1}} \!\!\frac{f (x'', x_n, \sqrt{1 + |x''|^2 + x_n^2})}{ \sqrt{1 + |x''|^2 + x_n^2}}\, d x''.\]
The right-hand side can be transformed to the same expression as follows.
\bea  I_r&=&\intl_{\bbr} e^{(1-n)t}dt\intl_{\bbr^{n-1}} f\left (e^{-t}v,\;\sinh t+\frac{|v|^{2}}{2}e^{-t},\;\cosh t+\frac{|v|^{2}}{2}e^{-t}\right )\, dv\nonumber\\
&{}& \mbox{\rm (set $e^t=r$, $\; v=rs\,\th$, $\;\th\in S^{n-2}$)}\nonumber\\
&=&\intl_0^\infty s^{n-2} ds \intl_0^\infty \frac{dr}{r}\intl_{S^{n-2}} f\left (s\th, \; \frac{r}{2}+\frac{r^2s^2-1}{2r},\; \frac{r}{2}+\frac{r^2s^2+1}{2r}\right )\, d\th\nonumber\\
&=&\intl_0^\infty s^{n-2} ds \intl_{\bbr}\frac{dx_n}{ \sqrt{1 + s^2 + x_n^2}}\intl_{S^{n-2}} \!\!f\left (s\th, x_n,  \sqrt{1 + s^2 + x_n^2}\right )\, d\th\!=\!I_l.\nonumber\eea
\end{proof}

\subsubsection{Horospheres }

In the hyperboloid model of the hyperbolic space, horospheres  $\hat \xi\subset \hn$ are defined as the cross-sections of the hyperboloid  $\hn$ by the hyperplanes of the form
$[{\bf x}, \xi]=1$, where $\xi\in\Gam_+$.

We denote by $\hat \Gam$ the set all horospheres in $\hn$. There is a one-to-one correspondence between the sets $\Gam_+$  and $\hat \Gam$.
 Since the group $G$ is transitive on $\Gam_+$, then it is transitive on $\hat \Gam$ and $(g\xi)\,\hat{}=g\hat \xi$
 for any $g\in G$.

\begin{proposition} \label {distance0} Let $a\in \hn$, $\,\xi \in \Gam_+$, $\,\hat\xi=\{x\in \hn:\, [x, \xi]=1\}\in \hat \Gam$. Then
\be\label {distance} d (a, \hat \xi)=|\log [a,\xi]|.\ee
\end{proposition}
\begin{proof} Let $g\in G$ be a hypebolic rotation such that $a=gx_0$ and $\xi=g\eta$, where $x_0=(0, \ldots, 0,1)$ and $\eta \in \Gam_+$ has the form $\eta=(0, \ldots,0,t,t)$ with some $t>0$. Then
\[d\equiv d (a, \hat \xi)=d (x_0, \hat \eta)=d (x_0,y),\]
where  $y\in  \hat \eta$ is the nearest point to $x_0$. The equation of the horosphere $\hat \eta$ is
$[x,\eta]=1$ or $x_{n+1}=x_n +1/t$. If $t<1$, then $1/t>1$ and $y=(0, \ldots,0,-\sh \,d, \ch \,d)$. Since $y\in \hat \eta$, then
$\ch\, d =-\sh\, d +1/t$, which gives $d=-\log t$. If $t>1$, then  $y=(0, \ldots,0, \sh\, d, \ch \,d)$ and we get $d=\log t$. To complete the proof, it remains to note that $t=[x_0,\eta]=[a,\xi]$.
\end{proof}

There is a one-to-one correspondence between the  points $\xi \in \Gam_+$  and the pairs $(t,\om)\in \bbr \times \sn$, so that $\xi\equiv \xi_{t,\om}=e^t b(\om)$, $b(\om)\!=\!(\om, 1)\!\in \!\Gam_+$. One can readily show that
\be\label {dBGBGe2}  e^t\!=\!\xi_{n+1}, \quad \cosh t\!=\! \frac{\xi^2_{n+1}+1}{2\xi_{n+1}}, \quad   \cosh t\pm 1\!=\!\frac{(\xi_{n+1}\pm 1)^2}{2\,\xi_{n+1}}.\ee
By Proposition \ref {distance0},
\be\label {distance2}  d(x_0, \hat \xi_{t,\om})=|t| \quad \forall \om \in \sn.\ee
Indeed, by (\ref{distance}), $ d(x_0, \hat \xi)=|\log e^t  [x_0, b(\om)]|= |\log e^t |=|t|$.

\begin{corollary}\label {distance1a} For each $x\in\hn$ and each $\om \in S^{n-1}$, there is a unique horosphere $\hat\xi$ passing through $x$
 and given by the point $\xi=e^{t}b(\omega) \in \Gamma_+$ with $t=-\log[x,b(\omega)]$.
\end{corollary}
\begin{proof} It suffices to show that $x\in \hat\xi$. By (\ref{distance}),
\[ d (x, \hat \xi)=|\log [x,\xi]|= |\log [x,\,b(\om) \exp (-\log [x,b(\om)) ]|= |\log \,1|=0.\]
Hence, $x\in \hat\xi$.
\end{proof}

For $x\in \hn$ and $\om \in \sn$, we denote
\begin{equation}\label{llvre1}
\left\langle x,\omega\right\rangle =-\log[x,b(\omega)].
\end{equation}

\begin{corollary}\label {distance1aA} If  $x\in\hn$,  $\om \in S^{n-1}$,  $\xi=e^{s+ \left\langle x,\omega\right\rangle} b(\om)$, then
\[ d (x, \hat \xi)=|s|.\]
If  $\xi\!=\!e^{s} b(\om)$ and $x\!=\!k_\om a_t n_v x_0$, where  $k_\om\!\in \!K$, $k_\om e_n\!=\!\om$, then, for all $v\!\in \!\bbr^{n-1}$,
\be\label {PPPBBB} d (x, \hat \xi)=|s-t|.\ee
\end{corollary}
\begin{proof} By (\ref{distance}) and (\ref{llvre1}),
\bea d (x, \hat \xi)&=&|\log [x,e^{s+ \left\langle x,\omega\right\rangle} b(\om)]|=|s+\left\langle x,\omega\right\rangle +\log [x,b(\om)]|\nonumber\\
&=&|s-\log [x,b(\om)]+ \log [x,b(\om)|=|s|.\nonumber\eea
For the second statement, owing to (\ref{horo coord}), we have
\bea &&d (x, \hat \xi)=|\log [k_\om a_t n_v x_0, k_\om e^s \xi_0]|=|s+\log [a_t n_v x_0, \xi_0]|\nonumber\\
&&=|s+\log [(e^{-t}v,\,\sinh t+\frac{|v|^{2}}{2}\,e^{-t},\,\cosh t+\frac{|v|^{2}}{2}\,e^{-t}), (0, \ldots, 0,1,1)]|
\nonumber\\
&&=|s+\log [-\sh t +\ch t)|=|s-t|.\nonumber\eea
\end{proof}

 The horosphere
\[\hat\xi_{0}=\{x:\, [x, \xi_{0}]=-x_n +x_{n+1}=1\}, \qquad  \xi_{0}=(0,\ldots,0,1,1),\]
is the basic. All other horospheres are obtained from $\hat\xi_{0}$ using hyperbolic rotations.
In the group-theoretic terms, horospheres  can be equivalently defined as translates of the orbit
$Nx_0 =\hat\xi_{0}$ under $G.$ Every horosphere has the form $ka_{t}Nx_0$ for
some $k\in K$ and $t\in\bbr$ ($t$ gives the signed distance of the
horosphere to the origin $x_0$); cf. (\ref{distance2}).  The
 subgroup $MN$ of $G$ leaves  the basic horosphere $Nx_0$ fixed. Hence, we have the
homogeneous space identification $\hat \Gam=G/MN.$  Each horosphere
$ka_{t}N x_0$ is  identified uniquely with the point
$\xi\in\Gam$ according to
\begin{equation}\label{llvre}
\xi=ka_{t}\xi_{0}=e^{t}\,k\xi_{0}
=e^{t}b(\omega), \qquad \xi_{0}
=(0,\ldots,0,1,1),
\end{equation}
where $\omega\!=\!k e_{n} \in S^{n-1}$, $
b(\omega)\!=\!k\xi_{0}\!=\! (\omega, 1)\! \in \!\Gamma_+$.

\begin{lemma}  \label {distance3} In terms of (\ref{llvre}), the invariant measure on $\Gamma_+$ is
 defined by the formula
 \be\label {distance4}  d\xi=c\, e^{(n-1)t}dtd\omega,\qquad c=\const,\ee
 $dt$ being the Lebesgue measure on
$\bbr$ and $d\omega$ the  surface measure on $S^{n-1}$.
\end{lemma}
\begin{proof} We invoke the  delta function language, according to which for any continuous one-variable function $\psi$,
  \[(\del,\psi)=\intl_\bbr \psi (s) \,\del (s)\, ds=\psi(0).\]
To give this equality precise meaning, let
$\om_\e$ be a bump function
\be\label {omee} \om_\e (s)=\left \{ \begin{array} {ll} \displaystyle{\frac{C}{\e}\, \exp\left (-\frac{\e^2}{\e^2-|s|^2}\right )},& |s|\le \e,\\
0,& |s|> \e.\\\end{array}\right .\ee
Here $C$ is chosen so that $\int_{\bbr}\om_\e (s)\, ds=1$, that is,
\[ C=\Bigg (\,\intl_{|s|<1}\exp\left (-\frac{1}{1-|s|^2}\right )\, ds\Bigg )^{-1}.\]
 Then
\be\label {delz} (\del, \psi)=\lim\limits_{\e\to 0} (\om_\e,  \psi)=
\lim\limits_{\e\to 0} \intl_{\bbr}\om_\e (s)\,\psi(s)\,ds=\psi(0).\ee

Following this formalism, we define  a constant multiple of $d\xi$   by the formula
\[ c \intl_{\Gamma_+} \vp (\xi)\, d\xi= 2\intl_{\bbr^{n+1}_+} \!\!\vp ({\bf y}) \,\del (||{\bf y}||^2 )\, d{\bf y}=
2 \, \lim\limits_{\e \to 0} \intl_{\bbr^{n+1}_+} \!\!\vp ({\bf y})\, \om_\e (||{\bf y}||^2)\,  d {\bf y}, \] where
$\bbr^{n+1}_+= \{{\bf y}\in \bbr^{n+1}:\, y_{n+1}>0\}$, $\vp\in C_c^\infty (\bbr^{n+1})$, $ \supp \,\vp \subset  \bbr^{n+1}_+$,
$\om_\e$ is  the bump function   (\ref{delz}). Then the result follows by  simple calculation.
\end{proof}

\begin {remark} {\rm For our purposes, we choose $c=\sig_{n-1}^{-1}$ in (\ref{distance4}), so that
\be\label {distance4a} \intl_{\Gamma_+} \vp (\xi)\, d\xi=\intl_{\sn}\intl_{\bbr} \vp (e^{t}b(\omega)) \,e^{(n-1)t}\,dt d_*\omega.\ee
 In particular, if $\vp$ is zonal, $\vp (\xi)=\vp_0 (\xi_{n+1})$, then
 \be\label {disRRR04a} \intl_{\Gamma_+} \vp (\xi)\, d\xi=\intl_0^\infty  \vp_0 (s)\, s^{n-2}\, ds.\ee
Indeed, by (\ref{distance4a}),
\bea
\intl_{\Gamma_+} \vp (\xi)\, d\xi&=&\!\intl_{\Gamma_+} \!\vp_0 ([\xi, e_{n+1}])\, d\xi\!=\!\intl_{\sn}\intl_{\bbr} \!\vp_0 ([e^{t}b(\omega), e_{n+1}]) \,e^{(n-1)t}\,dt d_*\omega\nonumber\\
&=&\!\intl_{\bbr} \!\vp_0 (e^{t})\,e^{(n-1)t}\,dt=\intl_0^\infty  \vp_0 (s)\, s^{n-2}\, ds.\nonumber\eea
}
\end{remark}

\subsection{Basic Properties of the Horospherical  Transform}

 We recall that $ b(\omega)=(\om,1)\in \Gam_+$, $\;\om \in \sn$;
\[
  x_0=(0,\ldots,0,1)\in \hn,   \qquad \xi_{0}
=(0,\ldots,0,1,1)\in \Gam_+.\]
For $\xi \in \Gam_+$, we denote by  $\hat \xi $  the  horosphere defined by
$\hat \xi =\{x\in\hn  : \,[x,\xi]=1 \}$. Given $x\in\hn$, let
 $\check x=\{\xi\in\Gam_+ :\, [x,\xi]=1 \}$ be the set of all
  points in $\Gam_+$  corresponding to the horospheres
 passing through $x$. For
 sufficiently good functions $f:\hn\rightarrow\bbc$ and
 $\varphi:\Gam_+\rightarrow \bbc$,
the horospherical Radon transform and its dual are defined by
\[
(\fr{H}   f)(\xi)=\intl_{\hat \xi} f(x)\, d_{\xi}x \; \; \; \;
\mathrm{and\;} \; \; \; \;  (\fr{H}^* \varphi) (x)=\intl_{\check x}
\varphi (\xi)\, d_{x} \xi,
\]
respectively. The precise meaning of these integrals is given in terms of the
horospherical coordinates. Specifically, if $\xi=e^{t}b(\omega)=e^{t} k\xi_{0}$, $k\in K$, then
\begin{equation}\label{horo transf}
(\fr{H}  f)(\xi) =\intl_{N} \!f(ka_{t}n x_0)\,dn  \!= \!\intl_{\bbr^{n-1}}  \!\!f(ka_{t}n_{v} x_0)\,dv.
\end{equation}
We can also write this expression in the form
\be\label{hREREsf}
(\fr{H}   f)(\xi) \equiv (\fr{H}_{\omega}f)(t)=\intl_{[x, b(\om)]=e^{-t}} f(x)\,d_{\om,t},\ee
which resembles the hyperplane Radon transform \cite{H11}. The following Fourier Slice Theorem follows immediately from (\ref{tag 3.3Bray}).

\begin{theorem} \label{Fourier Slice} If $f\in C_c^\infty (\hn)$, then
\be\label{hREREsf1}
\tilde f (\lambda, \om)=\intl_{\bbr} e^{-t(i\lambda-\del)}\,(\fr{H}_{\omega}f)(t)\, dt, \qquad \del=\frac{n-1}{2}.
\ee
\end{theorem}

To give the dual transform $\fr{H}^* \varphi$ precise meaning, let $x=\rho_x x_0, \; \rho_x \in G$, $\vp_x (\xi)=\vp (\rho_x \xi)$. Then we set
\begin{equation}\label{dual horo transf}
(\fr{H}^* \varphi) (x)=\intl_{K}\vp_x (k\xi_{o})\,dk=\intl_{\sn}\varphi_x(b(\omega))\,d_*\om.
\end{equation}
A more general expression
\be\label{coor dualGe00}
(\fr{H}_x^* \varphi) (t)\!=\!\!\intl_{S^{n-1}}\!\!\varphi_x(e^t b(\omega))\,d_*\om=\intl_{K}\vp_x (e^t k\xi_{o})\,dk, \quad t\in \bbr, \ee
is called the {\it shifted dual horospherical transform of $\vp$}. It averages $ \varphi$ over all horospheres at distance $|t|$ from $x$. The last observation is an immediate consequence of (\ref{PPPBBB}). Clearly, $(\fr{H}_x^* \varphi) (0)=(\fr{H}^* \varphi)(x)$.

The following statement  gives an alternative  representation of the dual transform. We recall the notation (cf. (\ref{llvre1})).
\[
\left\langle x,\omega\right\rangle =-\log[x,b(\omega)], \qquad x\in \hn, \quad \om \in \sn,  \quad b(\omega)=(\om,1)\in \Gam_+.
\]
\begin{lemma} {\rm (cf. \cite[Propossition 7.1]{Bru})}
For $x \in \hn$,
\begin{equation}\label{coor dual}
(\fr{H}^* \varphi) (x)\!=\!\!\intl_{S^{n-1}}\!\!e^{(n-1)\left\langle
x,\omega\right\rangle }\varphi(e^{\left\langle
x,\omega\right\rangle }
b(\omega))\,d_*\omega.
\end{equation}
\end{lemma}
\begin{proof} We write (\ref {coor dual}) in the equivalent form
\bea
(I_{1}\varphi) (g)&\equiv& \intl_{\sn}\varphi(g b(\omega))\,
d\omega \nonumber\\
\label{KK}&=&
\intl_{\sn} e^{(n-1)\left\langle
gx_0,\omega\right\rangle }
\varphi(e^{\left\langle
gx_0,\omega\right\rangle }  b(\omega))\, d\om \equiv  (I_{2}\varphi) (g).\qquad
\eea
Set $g=k^{\prime} a_{r}k^{\prime \prime} \ $
$ \ (k^{\prime}, \, k^{\prime \prime} \in K, \, a_{r} \in A)$. One can
readily see that $(I_{1}\varphi) (g)=(I_{2}\varphi) (g)$ if and only if
$(I_{1}\varphi^{\prime}) (a_{r})=(I_{2}\varphi^{\prime}) (a_{r})$, where $
\varphi^{\prime} (\xi)=\varphi (k^{\prime} \xi)$. Thus, it suffices to
prove (\ref {KK}) for $g=a_{r}$. Passing to polar coordinates on
$\sn$ and taking into account the equalities
\[
a_{r} e_{n}=(\cosh r)\, e_{n}+ (\sinh r)\, e_{n+1}, \quad
a_{r}e_{n+1}=(\sinh r) \,e_{n}+ (\cosh r)\, e_{n+1},
\]
 we have
\bea
&&(I_{1}\varphi) (a_{r}) =\intl_{-1}^{1} (1-\eta^2)^{(n-3)/2} d\eta\nonumber\\
 &&\times  \intl_{S^{n-2}}\varphi(\sqrt{1-\eta^2}\, \theta
+(\eta \cosh r + \sinh r)\, e_{n} + (\eta \,\sinh r + \cosh r)\, e_{n+1})\,
 d\theta,\nonumber\eea
\bea
\label{Itwo} (I_{2}\varphi) (a_{r}) & =&\intl_{-1}^{1} \frac {(1-\tau^2)^{(n-3)/2}}
{(\cosh r- \tau\, \sinh r)^{n-1}}\,  d\t\\
&\times&
   \intl_{S^{n-2}}\!\varphi \Big (\frac
{\sqrt{1-\tau^2} \,\theta+ \tau e_{n}+ e_{n+1}}{\cosh r-
\tau \,\sinh r} \Big ) \, d\theta. \nonumber
\eea
The  second expression can be  reduced to the first one if we put \[1/(\cosh r - \tau \sinh r)=\eta\, \sinh r + \cosh r.\]
\end{proof}

\begin{corollary} For $t\in \bbr$,
\be\label{coor dualGe}
(\fr{H}_x^* \varphi) (t)\!=\!\!\intl_{\sn}\!\!e^{(n-1)\left\langle
x,\omega\right\rangle }\,\varphi(e^{t+\left\langle
x,\omega\right\rangle }
b(\omega))\,d_*\omega.
\ee
\end{corollary}

Now we  establish basic properties of the operators $\fr{H}$ and $\fr{H}^*$.

\begin{proposition} \label {te for every} Let $f\in C_\mu (\hn)$, that is, $f \in C(\hn)$ and $f (x)=O (x_{n+1}^{-\mu})$. If $\mu>(n-1)/2$, then $(\fr{H} f)(\xi)$ is finite for every $\xi \in \Gam_+$.
\end{proposition}
\begin{proof} By (\ref{horo transf}) and (\ref{horo coord}),
\bea
|(\fr{H} f)(\xi)| &\le& c \intl_{\bbr^{n-1}}\frac{dv}{[a_{t}n_{v} x_0, x_0]^\mu}= c\,e^{-t(n-1)} \intl_{\bbr^{n-1}}\frac{dv}{[a_{t}n_{e^{-t} v} x_0, x_0]^\mu}\nonumber\\
&=& c\,e^{-t(n-1)} \intl_{\bbr^{n-1}}\frac{dv}{(\ch t + (|v|^2/2)\,e^{-t})^\mu}.\nonumber\eea
The last integral is finite whenever $\mu>(n-1)/2$.
\end{proof}

\begin{remark} {\rm The condition $\mu>(n-1)/2$ is sharp. There is a function $\tilde f\in C_\mu (\hn)$, $\mu\le (n-1)/2$, for which $(\fr{H} \tilde f)(\xi)\equiv \infty$. An example of such a function can be constructed using the formula (\ref{R form 1}) below. The question about the existence of $\fr{H} f$ for $f\in L^p(\hn)$ requires more preparation and will be answered in Proposition \ref{Lpestimate}.}
\end{remark}

\begin{lemma}
 Let $f$ and $\varphi$ be functions on $\hn$ and
$\Gam_+,$ respectively. Then the duality relation
\begin{equation}\label{dualityA}
\intl_{\Gam_+}\varphi(\xi)\, (\fr{H}  f)(\xi)\,d\xi=\intl_{\hn}(\fr{H}^* \varphi) (x)\,f(x)\,dx
\end{equation}
holds provided that  the integral in either side exists in the Lebesgue sense.
\end{lemma}
\begin{proof}  This statement is known in the general context of Radon transforms for double fibration; see Helgason \cite{H00}. For us it is important that the  definitions  of $\fr{H}$, $\fr{H}^*$, and measures in (\ref{dualityA})  are consistent. The direct proof is as follows.

Setting $\xi=e^{t}b(\omega),\, \omega=k e_{n}$, and using (\ref{distance4a}) and (\ref{horo transf}), we  write the left-hand side of (\ref {dualityA}) in the form
\be\label {the above}
\intl_{\bbr}e^{(n-1) t}\,dt\intl_{\sn}\varphi(e^{t}b(\omega))
\,d_*\omega\intl_{\bbr^{n-1}}f(ka_{t}n_{v}x_0)\,dv.
\ee
Put $v=e^{-t}u$,  $y=a_{t}n_{v}x_0=n_{u}a_{t}x_0$; cf. (\ref{normalizes}). Then (\ref{horo coord}) yields
\[ [ky,b(\omega)]=[kn_{u}a_{t}x_0, b(ke_n)]=[n_{u}a_{t}x_0, \xi_0]=e^{-t}.\]
This equality, combined with (\ref{kUUUPqs1}) and (\ref{llvre1}),
  allows us to write (\ref{the above})  as
\bea
 && \intl_{\sn}d_*\omega \intl_{\hn}
\varphi([ky,b(\omega)]^{-1} b(\omega))\, f(ky)\,[ky,b(\omega)]^{1-n}\,dy\nonumber\\
&&=   \intl_{\hn}  f(x)\,dx  \intl_{\sn}
\varphi(e^{\left\langle x,\omega\right\rangle } b(\omega))\,
e^{(n-1)\left\langle x,\omega\right\rangle }\,d_*\omega.\nonumber
\eea
Owing to (\ref{coor dual}), the latter coincides with the right-hand side of (\ref {dualityA}).
\end{proof}

The case of  zonal functions, when the  horospherical transform expresses through the Riemann-Liouville fractional integral
\be\label{rl-}
(I^\a_{-}g ) (r) = \frac{1}{\Gamma (\alpha)} \intl_r^{\infty}
\frac{g(s) \,ds} {(s-r)^{1- \alpha}}\, \qquad \a>0,\ee
is of particular importance. The following statement was proved in \cite[Lemma 7.3]{Bru}. We present it here for the sake of completeness. 

\begin{lemma} 
 Suppose that $f$ and $\varphi$ are locally integrable functions  on $\hn$ and
$\Gam_+$, respectively, and the integrals below exist in the Lebesgue sense.

\noindent {\rm (i)} If
$f$ is $K$-invariant,  $f(x)=f_{0}(x_{n+1}),$ then $\fr{H} f$
is $K$-invariant, and
\bea
\label{R form 1} (\fr{H}_{\omega}f)(t)\!\!&=& \!\!2^{(n-3)/2}\sig_{n-2}\, e^{t(1-n)/2}\intl_{{\rm cosh}
t}^{\infty}\!\!f_{0}(s)\,\left(s\!-\!\cosh t\right)^{(n-3)/2}ds\qquad \\
\label{R form 1a}\!\!&=&\!\! c_1\, e^{t(1-n)/2}\,(I_-^{(n-1)/2} f_0)(\ch t), \quad c_1=(2\pi)^{(n-1)/2}.
\eea

\noindent {\rm (ii)} If $\varphi$ is  $K$-invariant,
$\varphi(\xi)=\varphi_{0}(\xi_{n+1}),$ then $\fr{H}^* \varphi$ is
$K$-invariant, and
\begin{equation}\label{dual R}
(\fr{H}^* \varphi) (x)\!=\!\frac{c_2}{(\sinh r)^{n-2}}\!
\intl_{-r}^{r}\!\varphi_{0}(e^{s})\,(\cosh r\!-\!\cosh s)^{(n-3)/2}\,e^{s(n-1)/2} ds,
\end{equation}
\[ c_2=\frac{2^{(n-3)/2}\Gam (n/2)}{\pi^{1/2} \,\Gam ((n-1)/2)}, \qquad \cosh r=x_{n+1}.\]
\end{lemma}

\begin{proof}
(i) \ Since $f$ is $K$-invariant, then one can ignore $k$ in (\ref{horo transf}), and by
(\ref{horo coord}) we have
\bea
(\fr{H}_{\omega}f)(t) & =&\intl_{\bbr^{n-1}}f(a_{t}n_{v}x_0)\,dv=
e^{(1-n) t} \intl_{\bbr^{n-1}}  f(n_{v}a_{t}x_0)\,dv \nonumber\\
& =&e^{(1-n) t}\intl_{\bbr^{n-1}}  f_{0}\big (\cosh t+
\frac{|v|^{2}}{2}\, e^{-t} \big )\, dv.\nonumber
\eea
This gives (\ref{R form 1}).

(ii) \ By making use of (\ref {Itwo}), we obtain

\bea
&&(\fr{H}^* \varphi) (x)=\frac{\sig_{n-2}}{\sig_{n-1}} \intl_{-1}^{1} (1-\tau^{2})^{(n-3)/2}\nonumber\\
&&\times \,(\cosh r -\tau \sinh r)^{1-n} \,\varphi_{0} \!
 \left (\!\frac{1}{\cosh r- \tau \sinh r} \right )\, d\tau \nonumber\\
&& =\frac {\sig_{n-2}}{\sig_{n-1}\,(\sinh r)^{n-2}}\intl_{-r}^{r}
(\sinh^{2} r \!-\! (\cosh r \!-\! e^{-s})^{2})^{(n-3)/2} e^{s(n -2)}
 \varphi_{0}(e^{s}) \; ds,\nonumber
\eea
which gives  (\ref{dual R}).
\end{proof}
\begin{remark} {\rm The operator $\fr{H}^*$ is non-injective on the class of all functions on which it is well-defined. If, for instance,  $\varphi(\xi)=\varphi_{0}(\xi_{n+1}),$ where the function $\tilde \vp (s)=\varphi_{0}(e^{s})\,e^{s(n-1)/2}$ is odd, then $\fr{H}^*\vp=0$. Take, for example $\varphi(\xi)=\xi_{n+1}^{1-n/2} -\xi_{n+1}^{-n/2}$, for which
\[\tilde \vp (s)=[e^{s(1-n/2)} -e^{-sn/2}]\,e^{s(n-1)/2}=2\,\sh\, (s/2).\]
}
\end{remark}

Below we give some examples, in which  $\del=(n-1)/2$.

\begin{example} {\rm Let $f(x)\!=\!x_{n+1}^{-\b}$,  $\b> \del$. Then
\[
(\fr{H}_{\omega}f)(t)\!= \!\frac{(2\pi)^{\del} \,\Gam (\b-\del)}{\Gam (\b)}\, e^{-\del t}\, (\ch t)^{\del -\b},\]
or, by (\ref{dBGBGe2}),
\be\label {llopqwXX}
(\fr{H} f)(\xi)= c_\b\,\xi_{n+1}^{\b-2\del}\,(\xi_{n+1}^2\! +\!1)^{\del -\b},  \qquad  c_\b=\frac{2^{\b} \pi^{\del} \,\Gam (\b-\del)}{\Gam (\b)}.\ee
}
\end{example}

\begin{example}\label {llopqw0UU} {\rm Let $Re \, \a>0$,
\[
\varphi(\xi)\!\equiv \!\vp_0 (\xi_{n+1})\!=\!\frac{|\xi_{n+1} -1|^{\a -1}}{ \xi_{n+1}^{\del+\a/2}}, \qquad  \psi(\xi)\!\equiv\! \psi_0 (\xi_{n+1})\!=\!\frac{|\xi_{n+1} -1|^{\a -1}}{ \xi_{n+1}^{\del+\a/2-1}}.\]
 Below we show that the dual transforms $\fr{H}^* \varphi$ and $\fr{H}^* \psi$ coincide.
 Indeed, setting $\xi_{n+1}=e^{s}$ and using (\ref{dBGBGe2}), for the first function we have
\[
\vp_0 (e^{s})=\frac{2^{(\a-1)/2}\, (\ch s -1)^{(\a-1)/2}}{e^{s(\del+1/2)}}.\]
Hence, (\ref{dual R}) yields
\bea
\label {llopqwXXe}(\fr{H}^* \varphi) (x)&=&\frac{2^{(\a-1)/2}\,c_2}{(\sinh r)^{2\del -1}}
\intl_{-r}^{r}(\cosh r\!-\!\cosh s)^{\del -1}\\
&\times& (\cosh s-1)^{(\a-1)/2}\, e^{-s/2}\, ds,\qquad \ch r=x_{n+1}.\nonumber\eea
 Using  properties of the hyperbolic functions, we continue:
\bea
(\fr{H}^* \varphi) (x)&=&\frac{2^{(\a+1)/2}\,c_2}{(\sinh r)^{2\del -1}}\intl_{0}^{r}(\cosh r\!-\!\cosh s)^{\del -1}\nonumber\\
&\times& (\cosh s-1)^{(\a-1)/2} \, \ch (s/2)\, ds\nonumber\eea
or
\be\label {lYYXe}(\fr{H}^* \varphi) (x)\!=\!\frac{2^{\a/2}\,c_2}{(\sinh r)^{2\del -1}}\intl_{0}^{r}(\cosh r\!-\!\cosh s)^{\del -1}(\cosh s-1)^{\a/2-1} \, \sh s\, ds. \ee
 In the similar expression for $\fr{H}^* \psi$, the exponent $e^{-s/2}$ in (\ref{llopqwXXe}) must be replaced by $e^{s/2}$, but all the rest remains unchanged.
 The  integral (\ref{lYYXe}) can be easily computed and we get
\be\label {llopqw0}
(\fr{H}^* \varphi) (x)=(\fr{H}^* \psi) (x)=c_\a \,\frac{(x_{n+1}\!-\!1)^{(\a-1)/2}}{(x_{n+1}\!+\!1)^{\del -1/2}}, \ee
\be\label {llopqw} c_\a=\frac{2^{\del \!+\!\a/2 \!-\!1}\,\Gam (\del+1/2)\, \Gam (\a/2)}{\pi^{1/2} \,\Gam (\del+\a/2)}.\ee
}
\end{example}

\begin{example} {\rm Let
\[
\varphi(\xi)\equiv \vp_0 (\xi_{n+1})= \frac{(\xi_{n+1} -1)^{\a -1}}{(\xi_{n+1}+1)^{\a-1+2\del}}, \qquad Re \, \a>0, \qquad \del\!=\!\frac{n\!-\!1}{2}.\]
 If $\xi_{n+1}=e^{s}$, then, by (\ref{dBGBGe2}),
\[
\vp_0 (e^{s})=\frac{2^{-\del}\,e^{-s\del}\, (\ch s -1)^{(\a-1)/2}}{(\ch s +1)^{(\a-1)/2+\del}}.\]
Hence, by (\ref{dual R}),
\bea
(\fr{H}^* \varphi) (x)&=&\frac{2^{1-\del}\,c_2}{(\sinh r)^{2\del -1}}
\intl_{0}^{r}\frac{(\cosh r\!-\!\cosh s)^{\del -1}(\cosh s-1)^{(\a-1)/2}}{ (\cosh s+1)^{(\a-1)/2+\del}}\, ds\nonumber\\
&=&\frac{2^{1-\del}\,c_2}{(\ch^2 r -1)^{\del -1/2}} \intl_1^{{\rm cosh}\, r} \frac{(\cosh r\!-\!z)^{\del -1} (z-1)^{\a/2-1}}{(z+1)^{\a/2+\del}}\, dz.\nonumber\eea
This integral can be computed using \cite[formula 2.2.6(2)]{PBM1},  and we get
 \be\label {llopqw1}
(\fr{H}^* \varphi) (x)=\tilde c_\a\, \frac{(x_{n+1} \!-\!1)^{(\a-1)/2}}{(x_{n+1} \!+\!1)^{(\a-1)/2+\del}}, \quad
\tilde c_\a\!=\!\frac{2^{-\del}\,\Gam (\del+1/2)\, \Gam (\a/2)}{\pi^{1/2} \,\Gam (\del+\a/2)}.
\ee
}
\end{example}

Combining these examples with the duality (\ref {dualityA}), we arrive at the following statement.
\begin{lemma}  For $Re \, \a>0$ and  $Re \, \b>\del=(n-1)/2$, the
 following equalities hold provided that the integrals in either side exist in the Lebesgue sense:
\begin{equation}\label{dualityA1a}
\intl_{\hn}(\fr{H}^* \varphi) (x)\,x_{n+1}^{-\b}\,dx =c_\b  \intl_{\Gam_+}\!\!\varphi(\xi)\,\xi_{n+1}^{\b-2\del}\,
(\xi_{n+1}^2\! +\!1)^{\del -\b}\,d\xi,
\end{equation}
\be\label{dualityA2a}
\intl_{\Gam_+}\! (\fr{H}  f)(\xi)\, \frac{|\xi_{n+1}\!-\! 1|^{\a-1}}{\xi_{n+1}^{\del+\a/2}}\, d\xi
\!=\!c_\a \intl_{\hn} \!f(x)\, \frac{(x_{n+1} \!-\!1)^{(\a-1)/2}}{(x_{n+1} \!+\!1)^{\del - 1/2}}\,dx,\ee
\be\label{dualityA2aold}
\intl_{\Gam_+}\! (\fr{H}  f)(\xi)\, \frac{|\xi_{n+1}\!-\! 1|^{\a-1}}{\xi_{n+1}^{\del+\a/2-1}}\, d\xi
\!=\!c_\a \intl_{\hn} \!f(x)\, \frac{(x_{n+1} \!-\!1)^{(\a-1)/2}}{(x_{n+1} \!+\!1)^{\del - 1/2}}\,dx,\ee
\be\label{dualityA3a} \intl_{\Gam_+}\! (\fr{H} f)(\xi)\, \frac{(\xi_{n+1} -1)^{\a -1}}{(\xi_{n+1}+1)^{\a-1+2\del}} \, d\xi\!=\!\tilde c_\a
\intl_{\hn} \!f(x)\, \frac{(x_{n+1} \!-\!1)^{(\a-1)/2}}{(x_{n+1} \!+\!1)^{(\a-1)/2+\del}}\,dx,\ee
where the constants $c_\b, \; c_\a, \; \tilde c_\a$ are defined by (\ref{llopqwXX}), (\ref{llopqw}), and (\ref{llopqw1}), respectively.
In particular, for $\a=1$,  (\ref{dualityA3a}) yields
\begin{equation}\label{dualityA4a}
\intl_{\Gam_+}\! (\fr{H} f)(\xi)\, \frac{d\xi}{(\xi_{n+1}+ 1)^{2\del}}=2^{-\del}\intl_{\hn}f(x)\,  \frac{dx}{(x_{n+1}\!+\!1)^{\del}}\,.
\end{equation}
\end{lemma}

The equalities (\ref{dualityA1a}) - (\ref{dualityA4a}) provide  precise information about the existence of the integrals $\fr{H} f$ and   $\fr{H}^* \varphi$.  For example, (\ref{dualityA4a}) gives the following result.

\begin{proposition}
\label{Lpestimate}If $f\in L^{p}(\hn),$  $1\leq p<2,$ then
$(\fr{H} f)(\xi)$ exists a.e. and
\begin{equation}\label{Lp estimate}
\intl_{\Gam_+}\! |(\fr{H} f)(\xi)|\, \frac{d\xi}{(\xi_{n+1}+ 1)^{n-1}} \le c\,\left\|  f\right\|_{p},
\end{equation}
 If $p\ge 2$, then there is a function $\tilde f \in L^{p}(\hn)$ such that $(\fr{H} f)(\xi)\equiv
\infty$.
\end{proposition}

\begin{proof} By H\"older's inequality, the integral on the right-hand side of (\ref{dualityA4a}) does not exceed $c\,\left\|  f\right\|_{p}$. Here,  by (\ref{ppooii}),
\[c^{p'}=\intl_{\hn} \frac{dx}{(x_{n+1}\!+\!1)^{\del p'}}= \sig_{n-1}
\intl_1^\infty  \frac{(s^2 -1)^{n/2 -1}\, ds}{(s\!+\!1)^{p' (n-1)/2}}<\infty\]
 if $1\le p <2$. If $p\ge 2$, then  the function
\[
 \tilde f (x)=\frac{(x_{n+1}^2-1)^{(1-n/2)/p}}{(x_{n+1}+1)^{1/p}\, \log (x_{n+1}+1)}\]
 belongs to $L^{p}(\hn)$. However, the integral (\ref{R form 1})  diverges for this function.
\end{proof}

The existence of $\fr{H} f$ for  $f\in L^{p}(\hn)$,  $1\leq p<2$, was first established in \cite{Bru}.

\section{Inversion Formulas}

In this section we obtain main results of the present paper.
 For a compactly supported smooth function $f$ one can write
\begin{equation}\label {Horo1}
(\fr{H} f)(\xi)=\intl_{\bbe^{n,1}}f(x)\,\delta([x,\xi]-1)\,dx,\qquad \xi \in \Gam_+.
\end{equation}
The following inversion formulas can be found in Gelfand, Graev, and Vilenkin \cite{GGV}; see also Vilenkin and Klimyk \cite[p. 162]{VK}):
\be\label {Horo2}
f(x)=\left \{\begin{array} {ll} \displaystyle{\frac{(-1)^{m}}{2(2\pi)^{2m}}
\intl_{\Gam_+}\delta^{(2m)}([x,\xi]-1)\,(\fr{H} f)(\xi)\,d\xi}   \, &\text{\rm if} \; n=2m+1,\\
{} \\
\displaystyle{\frac{(-1)^{m}\Gamma(2m)}{(2\pi)^{2m}}
\intl_{\Gam_+}([x,\xi]-1)^{-2m}\,(\fr{H} f)(\xi)\,d\xi} \, &\text{\rm if} \; n=2m. \\
\end{array}\right.
\ee
The divergent integrals in these formulas are given precise sense  in the framework of the theory of distributions.
In this section we  obtain alternative inversion formulas which do not contain divergent integrals and are applicable not only to smooth functions, but also to $f \in L^p (\hn)$.

\subsection{The Method of Mean Value Operators}

This method  relies on the implementation of the shifted dual horospherical  transform $\fr{H}^*_x \vp$  and the spherical mean  $(M_x f)(s)$; see (\ref{coor dualGe00}),  (\ref{2.21hDIF}).

\begin{lemma} \label {NNNDDEER} If $\vp=\fr{H} f$, then
\be\label {gyuyyhER} (\fr{H}^*_x \vp)(t)=(2\pi\, e^{-t})^{\del}
\, (I^{\del}_{-}  M_x f)(\ch t), \quad \del=(n-1)/2,\ee
where $I_-^{\del}f_0$ is  the Riemann-Liouville fractional integral (\ref{rl-}).
It is assumed  that the expression on either side of (\ref{gyuyyhER}) is finite when $f$ is
replaced by $|f|$.
\end{lemma}
\begin{proof} Fix $x\in \hn$ and let $f_x (y)= f(\om_x y)$, where $y\in \hn$, $\om_x \in G$, $\om_x x_0 =x$. By (\ref{coor dualGe00}), owing to $G$ invariance, we have
\[
(\fr{H}^*_x \vp)(t)=\intl_K \!(\fr{H} f) (\rho_x\,e^t k \xi_0)\, dk=\fr{H} \Big [\intl_K f_x (ky)\, dk \Big ] (e^t \xi_0).\]
The function $y \to  \intl_K f_x (ky)\, dk$ is zonal, so that there is a one-variable function $f_{0,x} (\cdot)$ such that
\be\label {gyuyyhYPER} f_{0,x} (y_{n+1})= \intl_K f_x (ky)\, dk.\ee
By (\ref{R form 1a}),
\[
(\fr{H}^*_x \vp)(t)=(2\pi)^{\del}\, e^{-t\del}\,(I_-^{\del}  f_{0,x})(\ch t), \]
where, by (\ref {gyuyyhYPER}),
\[
f_{0,x}(s)=\intl_K f_x (k(e_n \,\sqrt{s^2 -1} +e _{n+1}\, s ))\, dk=(M_x f)(s),\]
as desired.
\end{proof}

Following Lemma \ref{NNNDDEER}, we denote
\be\label {next aim} g_x (s)=(M_x f)(s), \qquad \psi_x (\t)\!=\!(2\pi\, e^{-t})^{-\del}(\fr{H}^*_{x} \fr{H} f)(t)\big |_{t= {\rm cosh}^{-1} \t}\,.\ee
Then (\ref{gyuyyhER}) can be written as
\be\label{uutrrhER}(I^{\del}_{-}g_x)(\t)=\psi_x (\t), \qquad \del=(n-1)/2.\ee

According to Lemma \ref {NNNDDEER}, to reconstruct $f$, we need to show that the natural assumptions for $f$, as in  Propositions \ref{te for every} and  \ref{Lpestimate},  guarantee the existence of  $I^{\del}_{-}g_x$ in the Lebesgue sense. Then we reconstruct $g_x (s)=(M_x f)(s)$ from the Abel-type equation (\ref{uutrrhER}). The function $f$ will be obtained as a limit $f(x)=\lim\limits_{s\to 1} (M_x f)(s)$ in a suitable sense.

To find $g_x (s)$ from  (\ref{uutrrhER}), we need to develop the pertinent tools of fractional differentiation.

The  Riemann-Liouville fractional derivative $\Cal D^\a_{-}$ is defined as the  left inverse of the corresponding operator $I^{\a}_{-}$ on a half-line $(a, \infty)$.   The next proposition is a slight modification of Lemma 2.1 from \cite{Ru13b}.
  \begin{proposition}\label{lif} Let $a>0$,
 \be\label{for10} \intl_a^\infty |g(s)|\, s^{\a -1}\, ds
<\infty, \qquad \a>0.\ee
The following statements hold.

\noindent {\rm (i)} For any  $\b\in [0,\a]$, the integral $(I^\b_{-}g)(s)$\footnote{In the case $\b=0$ we set $I^0_{-}g=g$.} is finite for almost all $s> a$ and
\[ (I^\a_{-} g)(s)=(I^\b_{-} I^{\a-\b}_{-} g)(s).\]

\noindent {\rm (ii)} If $\a\ge 1$, then  $(I^\a_{-} g)(s)$ is continuous on $(a, \infty)$.

\noindent {\rm (iii)} If $g$ is non-negative, locally integrable on $[a,\infty)$, but (\ref{for10}) fails, then $(I^\a_{-}g)(s)=\infty$ for every $s\ge a$.
\end{proposition}

The analytic form of the fractional derivative $\Cal D^\a_{-} : \,I^\a_{-}g \to g$ is determined by the behavior of $g$ at infinity and the value of $\a$.  Here some traditional inversion methods may not work. Suppose, for example, that we want to reconstruct a function $g\in L^p (a,\infty)$ from   $I^{\alpha}_{-} g=h$. We assume $1\le p<1/\a$, so that   $I^{\alpha}_{-} g$ is well-defined. A standard Riemann-Liouville inversion procedure  yields
\[(\Cal D^\a_{-}I^{\alpha}_{-} g)(s)=(-d/d s)\, (I^{1 - \alpha}_{-}I^{\alpha}_{-} f)(s)=(-d/d s)\,
  (I^{1}_{-} g)(s).\]
The integral $(I^{1}_{-} g)(s)=\int_s^\infty g(\eta)\,d\eta$ is convergent if
$g\in L^1 (a,\infty)$, but it may not exist if   $p>1$.

To circumvent this difficulty, we invoke compositions with power functions.

In the next statement,  the powers of $s$ stand for the corresponding multiplication operators.

\begin{theorem}\label{78awqeTT} Let $(I^\a_{-} g)(s)=h(s)$, $\a>0$,   and suppose that $g$ satisfies (\ref{for10}). Then  $g(s)= (\Cal D^\a_{-} h)(s)$ for almost all $s>a$, where  $\Cal D^\a_{-} h$ has one of the following forms.

\noindent {\rm (i)} If $\a=m$ is an integer, then
\be\label {90bedrTT}
\Cal D^\a_{-} h=(- 1)^m h^{(m)}.\ee

\noindent {\rm (ii)}   If $\alpha = m +\alpha_0, \; m = [ \alpha], \; 0 < \alpha_0 <1$, then
\be\label{frYUT}  \Cal D^\a_{-} h=(-1)^{m+1}  s^{1-\a_0}\, \left [ s^{m-j+\a_0 }\, I^{1-\a_0}_{-} \,s^{j-m-1 }\, h^{(j)}\right ]^{(m-j+1)}\ee
with any $j=0,1, \ldots, m$. Under the stronger condition
\be\label{for10TTU} \intl_a^\infty |g(s)|\, s^{m}\, ds
<\infty,\ee
 we have
\be\label{frr+z4y0TT} \Cal D^\a_{-} h=(-1)^{m+1} [I^{1-\a+m}_{-}\, h]^{(m+1)}.\ee
\end{theorem}

\begin{proof} Let  $\a=m$ be an integer.   Then  $I^m_{-} g=I^1_{-} I^{m-1}_{-} g$  and (\ref{90bedrTT}) can be obtained by consecutive  differentiation. To establish (\ref{frYUT}), we make use of  the composition formula
\be\label {90bedrTT54} I^{\mu+\nu}_{-}\, s^{-\nu}g=s^{\mu }\, I^\nu_{-} \,s^{-\mu-\nu }\,I^\mu_{-}g, \qquad \mu,\nu>0,\ee
which can be easily proved by changing the order of integration. If
\[
\intl_a^\infty |g(s)|\, s^{\mu -1}\, ds
<\infty,\]
then, by Proposition \ref{lif}, the integrals  $I^{m-1}_{-} g$ and $I^{\mu+\nu}_{-}\, s^{-\nu}g$ exist simultaneously and application of Fubini's theorem in  (\ref{90bedrTT54}) is well-justified.
The parameters  $\mu$ and $\nu$ will be chosen according to our needs.

 Assuming $\alpha = m +\alpha_0, \; m = [ \alpha], \; 0 < \alpha_0 <1$, we write  $I^\a_{-} g=h$  as
\[I^j_{-} I^{m-j+\a_0}_{-} g =h\] for any $j=0,1, \ldots, m$ (here we  use Proposition \ref{lif} again). Hence, the differentiation yields
\be\label{vbty5} I^{m-j+\a_0}_{-} g =(-1)^jh^{(j)}.\ee
Setting $\mu=m-j+\a_0$, $\;\nu=1-\a_0$ in (\ref{90bedrTT54}), we obtain
\[
I^{m-j+1}_{-} s^{\a_0-1}g=s^{m-j+\a_0 }\, I^{1-\a_0}_{-} \,s^{j-m-1 }\,I^{m-j+\a_0 }_{-}g.\]
By (\ref{vbty5}), it follows that
\[ g=(-1)^{m+1}  s^{1-\a_0}\, \left [ s^{m-j+\a_0 }\, I^{1-\a_0}_{-} \,s^{j-m-1 }\, h^{(j)}\right ]^{(m-j+1)},\]
as desired.
 The formula (\ref{frr+z4y0TT}) follows from the equality
$I^{1-\a_0}_{-}I^{m +\alpha_0}_{-}g=I^{m +1}_{-}g $. The latter is
  is well-justified because  (\ref{for10TTU}) guarantees the existence of $I^{m +1}_{-}g$ in the Lebesgue sense.
\end{proof}

The following particular cases are especially  useful. Setting $j=0$ and $j=m$ in (\ref{frYUT}), for $\a=m+\a_0$ we  obtain
\bea\label{frYUT1}  \Cal D^\a_{-} h&=&(-1)^{m+1}  s^{1-\a_0}\, \left [ s^{m+\a_0 }\, I^{1-\a_0}_{-} \,s^{-m-1 }\, h\right ]^{(m+1)},\\
\label{frYUT2}  &=&(-1)^{m+1}  s^{1-\a_0}\,  \left [ s^{\a_0 }\, I^{1-\a_0}_{-} \,s^{-1 }\, h^{(m)}\right ]'.\eea
 If, for instance,  $\a=k/2$ and $k$ is odd,  then
\bea\label{frYUT1k}  \Cal D^{k/2}_{-} h&=&(-1)^{(k+1)/2}  s^{1/2}\, \left [ s^{k/2}\, I^{1/2}_{-} \,s^{-(k+1)/2}\, h\right ]^{((k+1)/2)},\\
\label{frYUT2k}  &=&(-1)^{(k+1)/2}  s^{1/2}\,  \left [ s^{1/2}\, I^{1/2}_{-} \,s^{-1 }\, h^{((k-1)/2)}\right ]'.\eea

%Now we apply the above auxiliary statements to our equation (\ref{uutrrhER}).

\begin{lemma}\label{jHYHIRC}  Let $ g_x (s)=(M_x f)(s)$, $a>1$,
\be\label{for10BG} I_\a (x)=\intl_a^\infty |g_x(s)|\, s^{\a -1}\, ds, \qquad \a>0.\ee
 If $f\in C_\mu (\hn)$, $\mu>\a$, or $f\in L^p (\hn)$, $1\le p<(n-1)/\a$, then $I_\a (x)<\infty$ for all $x\in \hn$.
 \end{lemma}
\begin{proof} It suffices to assume $f\ge 0$. Let $s=\ch r$, $A=\ch^{-1} a>0$. Changing variables and using (\ref{tag 2.17-HYP}), we obtain
\bea
I_\a (x)&=&\intl_A^\infty (M_x f)(\ch r)\, \ch^{\a -1} r\,\sh r\, dr\nonumber\\
&=&\sig_{n-1}^{-1}\intl_{y_{n+1}>A} f_x (y)\,\frac{y_{n+1}^{\a-1}}{(y_{n+1}^2-1)^{n/2 -1}}\, dy.\nonumber\eea
Here $f_x (y)=f (\om_x y)$, $\om_x \in G$ being a hyperbolic rotation that takes $e_{n+1}$ to $x$. For some $q\in [1,\infty]$, which will be specified later, by H\"older's inequality we have
$I_\a (x)\le \sig_{n-1}^{-1}  V^{1/q}W^{1/q'}$, $1/q+1/q'=1$, where $V=||f_x||_q^q= ||f||_q^q$,

 \bea W&=&\intl_{y_{n+1}>A}\frac{y_{n+1}^{(\a-1)q'}}{(y_{n+1}^2-1)^{(n/2 -1)q'}}\,dy\nonumber\\
&=& \sig_{n-1}\intl_{A}^\infty s^{(\a-1)q'}  \, (s^2-1)^{(n/2-1)(1-q')}\,ds .\nonumber\eea
If $q<(n-1)/\a$, then $W<\infty$.

Suppose that $f\in L^p (\hn)$. In this case, we choose $q=p$. Then  $V= ||f||_p^p$ and therefore, $||I_\a||_\infty <\infty$ provided that $p<(n-1)/\a$. If $f\in C_\mu (\hn)$, then
\[V= ||f||_q^q \le c\intl_{\hn}\frac{dy}{y_{n+1}^{\mu q}}=c\,\sig_{n-1}\intl_1^\infty \frac{(s^2 -1)^{n/2 -1}}{s^{\mu q}}\, ds.\]
This integral is finite whenever $q>(n-1)/\mu$. Thus, we can choose
\[\frac{n-1}{\mu} <q<\frac{n-1}{\a}\]
to get both $V$ and $W$ finite. If $\mu>\a$, such a $q$ exists.
\end{proof}

Setting $\a=\del=(n-1)/2$ in Lemma \ref{jHYHIRC} and using Proposition \ref{lif}, we obtain the following
\begin{corollary}
 Let $\del=(n-1)/2$, $n\ge 2$, $ g_x (s)=(M_x f)(s)$,
 where $f\in C_\mu (\hn)$, $\mu>\del$, or $f\in L^p (\hn)$, $1\le p<2$. Then
  the integral $(I^\del_{-}g_x)(s)$ exists in the Lebesgue sense for almost all $s> 1$ and all $x\in \hn$. If, moreover, $n\ge 3$, then $(I^\del_{-}g_x)(s)$ is a continuous function on $(1,\infty)$  for all $x\in \hn$.
\end{corollary}

\begin{lemma}\label{jjYBYB} Let $ g_x (s)=(M_x f)(s)$. Suppose that  $f\in C_\mu (\hn)$, $\mu>(n-1)/2$ or $f\in L^p(\hn)$, $1\le p<2$.
 If $I^\del_{-}g_x=\psi_x$,   as in (\ref{uutrrhER}), then
 \be\label {IO8kVF}g_x (s)= (\Cal D^\del_{-} \psi_x)(s)\qquad \forall \, s>1,\ee
 where $\Cal D^\del_{-} \psi_x$ is defined  as follows.

 \noindent {\rm (i)} If $n$ is odd, $n=2m+1$, then
\be\label{uimuyHOR}
(\Cal D^\del_{-} \psi_x)(s)= (-1)^{m } \psi_x^{(m)}(s). \ee

\noindent {\rm (ii)} If $n$ is even, $n=2m$, then
\bea\label{uimuyHOR2}
(\Cal D^\del_{-} \psi_x)(s)&=& (-1)^{m }   s^{1/2}\, \left [ s^{m-1/2}\, I^{1/2}_{-} \,s^{-m}\, \psi_x \right ]^{(m)},\\
\label{frYUT2m}  &=&(-1)^{m}  s^{1/2}\,  \left [ s^{1/2}\, I^{1/2}_{-} \,s^{-1 }\,\psi_x^{(m-1)}\right ]'.\eea
 Under the stronger assumptions $\mu >n/2$ or $1\le p<2(n-1)/n$,
 $(\Cal D^\del_{-} \psi_x)(s)$ can also be defined  by
  \be\label{uimuyHOR3}
(\Cal D^\del_{-} \psi_x)(s)=(-1)^{n/2} \left [ (I^{1/2}_{-}\psi_x)(s)\right ]^{(n/2)}.\ee
The equalities (\ref{uimuyHOR})-(\ref{uimuyHOR3})  hold for all $x\in \hn$, if  $f\in C_\mu (\hn)$, and for almost all $x\in \hn$, if  $f\in L^p(\hn)$.
\end{lemma}
\begin{proof}  We apply  Lemma \ref{jHYHIRC} and Theorem \ref{78awqeTT}.   The latter guarantees $ g_x (s)=(\Cal D^\del_{-} \psi_x)(s)$ only for {\it almost all} $s\ge 1$. Since, by Lemma \ref{Lemma 2.1-HYP}, $ g_x (s)=(M_x f)(s)$ is a continuous function of both $s$ and $x$, the result follows.
If $f\in L^p(\hn)$, $1\le p<2$, then by Lemma \ref{jHYHIRC} (with $\a=\del$),
\[\intl_a^\infty |g_x(s)|\, s^{\del -1}\, ds<\infty\]
 for all $a>1$ and all $x\in \hn$. Hence, by Theorem \ref{78awqeTT}, $ g_x (s)=(\Cal D^\del_{-} \psi_x)(s)$ is defined by
 (\ref{uimuyHOR})-(\ref{frYUT2m}) for all $x\in \hn$  and almost all $s\ge 1$.  However, by Lemma \ref{Lemma 2.1-HYP}, $ g_x (s)=(M_x f)(s)$ is a continuous $L^p$-valued function of  $s$. It follows that (\ref{uimuyHOR})-(\ref{frYUT2m}) extend to all $s> 1$, but for almost all $x\in \hn$. The proof of (\ref {uimuyHOR3}) is similar.
\end{proof}

Lemma \ref{jjYBYB} implies the following inversion result. We set
\[ \vp =\fr{H} f, \qquad \psi_x (s)\!=\!(2\pi\, e^{-t})^{(1-n)/2}(\fr{H}^*_{x} \vp)(t)\big |_{t= {\rm cosh}^{-1} s}\, ,\]
where $(\fr{H}^*_{x} \vp)(t)$ is the shifted dual horospherical  transform (\ref{coor dualGe00}).

\begin{theorem} \label{jOOOthERC2} Let $f\in C_\mu (\hn)$  or $f\in L^p(\hn)$. If $\mu>(n-1)/2$, $1\le p<2$, then
\be\label{jOOOthERC3}  f(x)=\lim\limits_{s\to 1}(\Cal D^\del_{-} \psi_x)(s),   \qquad \del=(n-1)/2,\ee
where $(D^\del_{-} \psi_x)(s)$ is defined by (\ref{uimuyHOR})-(\ref{frYUT2m}).

If $\mu\! >\!n/2$, $1\!\le \!p\!<\!2(n\!-\!1)/n$, then $(D^\del_{-} \psi_x)(s)$ can be defined by (\ref {uimuyHOR3}). The limit in  (\ref{jOOOthERC3}) is uniform for $f\in C_\mu (\hn)$  and is understood in the $L^p$-norm if $f\in L^p(\hn)$.
\end{theorem}

\subsection{Fractional Integrals of the Semyanistyi Type and Polynomials of the Beltrami-Laplace Operator}\label{jOMMMC2}

The desired form of  fractional integrals associated to the horospherical Radon transform can be found if we replace  $f$ in   (\ref{dualityA2a}) and  (\ref{dualityA2aold}) by the shifted function
 $f_x (y)= f(\om_x y)$, where $x\in \hn$ is fixed and $\om_x \in G$ takes the origin $x_0=(0, \ldots, 0,1) \sim e_{n+1}$ to $x$.
For $Re \, \a>0$ we obtain
\bea
&&\intl_{\Gam_+}\! (\fr{H}  f)(\xi)\, \frac{|[x,\xi]\!-\! 1|^{\a-1}}{[x,\xi]^{(n+\a-1)/2}}\, d\xi=\intl_{\Gam_+}\! (\fr{H}  f)(\xi)\, \frac{|[x,\xi]\!-\! 1|^{\a-1}}{[x,\xi]^{(n+\a-3)/2}}\, d\xi\nonumber\\
\label{duLLLa}&&=\frac{2^{(n+\a-3)/2}\,\Gam (n/2)\, \Gam (\a/2)}{\pi^{1/2} \,\Gam ((n+\a-1)/2)} \intl_{\hn} \!f(x)\, \frac{([x,y] \!-\!1)^{(\a-1)/2}}{([x,y] \!+\!1)^{n/2-1}}\,dx.  \qquad \eea
In particular, for $\a=1$,
\be\label{duUII2a}
\intl_{\Gam_+}\!  \frac{(\fr{H}  f)(\xi)}{[x,\xi]^{n/2}}\, d\xi=\intl_{\Gam_+}\! \frac{(\fr{H}  f)(\xi)}{[x,\xi]^{n/2-1}}\, d\xi=
2^{n/2-1} \intl_{\hn} \frac{f(x)}{([x,y] \!+\!1)^{n/2-1}}\,dx.  \ee
Invoking the potential operator (\ref{BGBBQQ}) and excluding the values $  \a =1,3,5, \ldots\,$, we
 write (\ref{duLLLa}) as
\bea\label{mm983XX1}  && \gam_\a \intl_{\Gam_+}\! (\fr{H}  f)(\xi)\, \frac{|[x,\xi]\!-\! 1|^{\a-1}}{[x,\xi]^{(n+\a-1)/2}}\, d\xi\\
&&=\gam_\a \intl_{\Gam_+}\! (\fr{H}  f)(\xi)\, \frac{|[x,\xi]\!-\! 1|^{\a-1}}{[x,\xi]^{(n+\a-3)/2}}\, d\xi= (Q^{\a+n-1} f)(x),\nonumber\eea
\[ \gam_\a=\frac{ \pi^{(1-n)/2}\,\Gam ((1-\a)/2)}  {2^{\a+n-1}\,\Gam (n/2)\, \Gam (\a/2)}, \qquad  Re \, \a >0; \quad \a \neq 1,3,5, \ldots\, .\]
This formula suggests  to define the following  fractional integrals
\be\label{mm983XX2e}  (\fr{H}_i^\a  f)(\xi)\!= \!\intl_{\hn} \!f(x)\, h_{\a,i}([x,\xi])\, dx, \ee
\be\label{mm983XX2} (\stackrel{*}{\fr{H}}{\!}_i^\a\vp)(x)\!=\! \intl_{\Gam_+} \!\vp(\xi)\,h_{\a,i}([x,\xi])\, d\xi,\ee
where $i=1,2$,
\[ h_{\a,1}(s)=\gam_\a \,\frac{|s -1|^{\a -1}}{s^{(n+\a-1)/2}}, \qquad   h_{\a,2}(s)=\gam_\a \,\frac{|s -1|^{\a -1}}{s^{(n+\a-3)/2}}.\]
 Thus, we have proved the following statement which resembles known facts for the totally geodesic Radon transforms; cf. \cite[formula (4.5)]{Ru02c}.
\begin{lemma} \label{duJJYPhor} Let $Re\, \a>0$, $\a\neq 1,3,5, \dots $. Then
\be\label{oOOFRhor}
\stackrel{*}{\fr{H}}{\!}_i^\a \fr{H}  f = Q^{\a +n-1} f, \qquad i=1,2,\ee
provided
that either side of this equality exists in the Lebesgue sense.
\end{lemma}

We will need the  following auxiliary statement.
\begin{lemma} \label{duJJYPhor2} Let $h$ be a measurable function on $\bbr_+$. Suppose that the integrals
\be\label{mm983XX2A }  (H f)(\xi)= \intl_{\hn} \!f(x)\, h([x,\xi])\, dx, \quad (H^* \vp)(x)= \intl_{\Gam_+} \!\vp(\xi)\, h([x,\xi])\, d\xi\ee
exist in the Lebesgue sense. Then
\bea\label{mm983XX2A1}
(H f)(\xi)&=& \intl_{\bbr} (\fr{H}_\om f)(t) \,h (e^{s-t})\, e^{(1-n)t}\, dt, \quad \xi =e^s b(\om),\\
\label{mm983XX2A2}
(H^* \vp)(x)&=& \intl_{\bbr}  (\fr{H}^*_s \vp)(x)\,  h (e^{s})\, e^{(n-1)s}\, ds,\eea
where $(\fr{H}_\om f)(t)$ is the horospherical transform (\ref{horo transf}) and $(\fr{H}^*_s \vp)(x)$ is the shifted dual transform (\ref{coor dualGe}).
\end{lemma}
\begin{proof} Let $\om=ke_n$, $k\in K$. Passing to the horospherical coordinates (\ref{horo coord}), we obtain
\[(H f)(\xi)=\intl_{\bbr}\intl_{\bbr^{n-1}} f(ka_t n_v x_0)\, h([a_t n_v x_0, e^{s} \xi_0])\, e^{(1-n)t}\, dvdt.\]
As in the proof of Corollary  \ref{distance1aA}, owing to (\ref{horo coord}), we have $[a_t n_v x_0, e^{s} \xi_0]=e^{s-t}$. Hence,
\[(H f)(\xi)=\intl_{\bbr} h (e^{s-t})\, e^{(1-n)t}\, dt \intl_{\bbr^{n-1}} \!\!f(ka_t n_v x_0)\, dv,\]
which gives (\ref{mm983XX2A1}).
Further, if  $\xi =e^t b(\om)$, then (\ref{llvre1}) yields
\[[x,\xi]=e^{t-\left\langle x,\omega\right\rangle}, \qquad   \left\langle x,\omega\right\rangle =-\log[x,b(\omega)]. \]
Hence, by
 (\ref{distance4a}),
\bea (H^* \vp)(x)&=& \intl_{\bbr}e^{(n-1)t}\, dt \intl_{\sn}\!
\!\vp(e^t b(\om)) \,h(e^{t-\left\langle x,\omega\right\rangle})\, d_*\om\nonumber\\
&=& \intl_{\bbr}   h(e^{s})\, e^{(n-1)s}\, ds\intl_{\sn} \!\!\vp(e^{s+ \left\langle x,\omega\right\rangle} b(\om))\, e^{(n-1)\left\langle x,\omega\right\rangle}\,d_*\om.\nonumber\eea
By (\ref{coor dualGe}), the last expression coincides with (\ref{mm983XX2A2}).
\end{proof}

\begin{lemma} \label {MMM-HYP3hor}  Let $f$ and $\vp$ be compactly supported continuous functions on $\hn$ and $\Gam_+$, respectively. Then, for $i=1,2$,
\bea\label {tag 4.4Brayhor}
\lim\limits_{\a\to 0} \,(\fr{H}^\a_i f)(\xi) &=& \lam_n\, e^{(1-n)s}\,(\fr{H}_\om f)(s), \quad \xi\!=\!e^s b(\om)\in \Gam_+,\\
\lim\limits_{\a\to 0} \,(\fr{H}^\a_i f)(\xi) &=& \lam_n\, e^{(1-n)s}\,(\fr{H}_\om f)(s), \quad \xi\!=\!e^s b(\om)\in \Gam_+,\\
\label {tag 4.4Brayhor1}\lim\limits_{\a \to 0} (\stackrel{*}{\fr{H}}{}_i^\a\vp)(x) &=&  \lam_n\, (\fr{H}^*\vp)(x),\qquad x\in \hn;\eea
\be \label {SEtag 4.4Brayhor}\lam_n=\frac{2^{1-n}\, \pi^{1-n/2}}{\Gam (n/2)}.\ee
\end{lemma}
\begin{proof}  Both equalities  follow from  (\ref{mm983XX2e}) and (\ref{mm983XX2}), owing to Lemma \ref{duJJYPhor2}. For example,
\[
(\fr{H}^\a_1 f)(\xi)\!=\!\gam_\a  \intl_{\bbr} (\fr{H}_\om f)(t) \,\frac{|e^{s-t} \!-\!1|^{\a -1}}{e^{(s-t)(n+\a-1)/2}}\, e^{(1-n)t}\, dt=
\frac{1}{\gam_1(\a)}\intl_{\bbr} \frac{a_s(\a,z)}{ |z|^{1-\a}}\, dz,\]
where $\gam_1(\a)=2^\a\pi^{1/2}\Gamma(\a/2)/\Gamma((1-\a)/2)$,
\[a_s(\a,z)=\lam_n\, \left |\frac{e^{z}\! -\!1}{z}\right|^{\a -1} (\fr{H}_\om f)(s\!-\!z) \,
\exp \left((1-n)s +\frac{z(n-\a-1)}{2}\right ).\]
  Passing to the limit as $\a \to 0$, we obtain
\[ \lim\limits_{\a\to 0} \,(\fr{H}^\a_1 f)(\xi)=\lam_n  a_s(0,0)=\lam_n e^{(1-n)s}\,(\fr{H}_\om f)(s),\]
 as desired. For other operators the proof is similar.
\end{proof}

The next statement contains a horospherical analogue of the celebrated Fuglede formula for Radon-John transforms over  planes in $\rn$ \cite{F}. 
\begin{lemma} \label {MMM-HYP3horT}  For all $n\ge 2$, the following equality holds provided that either side of it exists in the Lebesgue sense:
\be\label {taRRRrayhor} (\fr{H}^*\fr{H} f)(x)= \lam_n^{-1}\, (Q^{n-1} f)(x), \ee
$\lam_n$ being the constant (\ref{SEtag 4.4Brayhor}).
\end{lemma}
\begin{proof}  For sufficiently good $f$, one can formally obtain (\ref{taRRRrayhor}) letting $\a\to 0$ in (\ref{oOOFRhor}) and using (\ref{tag 4.4Brayhor1}). A direct proof under  minimal assumptions for $f$ is the following.
Let $x=\om_x x_0$, $\om_x \in G$, $f_x (y)=f(\om_x y)$. By (\ref{dual horo transf})  and (\ref{horo transf}), owing to $G$-invariance, we have
\bea
&&(\fr{H}^*\fr{H} f)(x)=\intl_{K} (\fr{H} f_x)(k\xi_{o})\,dk=\fr{H} \, \Bigg [\, \intl_{K} f_x (ky)\, dk\Bigg ](\xi_0)\nonumber\\
&&=\intl_{\bbr^{n-1}} dv\intl_{K} f_x (kn_{v} x_0)\,dk\qquad \text{\rm (use (\ref{horo coord}))}\nonumber\\
&&=\intl_{\bbr^{n-1}} dv\intl_{K} f_x \left(k \left (v+\frac{|v|^2}{2}\, e_n\right )+\left (1+\frac{|v|^2}{2}\right )e_{n+1}\right )  dk\nonumber\\
&&=\intl_0^\infty r^{n-2} dr \intl_{S^{n-2}}\!\!d\sig \intl_{K} f_x \left(k \left (r\sig +\frac{r^2}{2}\, e_n\right )+\left (1+\frac{r^2}{2}\right )e_{n+1}\right ) dk.\nonumber\eea
Now we replace integration over $S^{n-2}$ by the  integration over the corresponding group $M=SO(n-1)$ and then change the order of integration. Changing variables, we get
\bea
(\fr{H}^*\fr{H} f)(x)&=&\sig_{n-2}\intl_0^\infty r^{n-2} dr \nonumber\\
&\times&\intl_{K} \!f_x \left(k \left (re_{n-1} +\frac{r^2}{2}\, e_n\right )+\left (1+\frac{r^2}{2}\right )e_{n+1}\right ) dk.\nonumber\eea
Noting that $e_{n-1} +(r/2)\, e_n=\sqrt {1+r^2/4} \,\eta$ for some $\eta\in \sn$, we continue:
\bea
&&(\fr{H}^*\fr{H}  f)(x)=\frac{\sig_{n-2}}{\sig_{n-1}}\intl_0^\infty r^{n-2} dr \intl_{\sn}  f_x (r\sqrt {1+r^2/4} \,\theta+(1+r^2/2)\,e_{n+1})\, d\theta\nonumber\\
&&\label {MMM-HYP3horT1}=c_1 \intl_1^\infty (t-1)^{(n-3)/2}\,dt  \intl_{\sn}  f_x (\sqrt {t^2 -1}\, \th +t\,e_{n+1})\, d\theta, \eea
\[
c_1=\frac{2^{(n-3)/2}\,\pi^{-1/2}\, \Gam (n/2)}{\Gam ((n-1)/2)}.\]
On the other hand, by (\ref{BGBBQQ}),
\bea &&(Q^{n-1} f)(x)=\z_{n,n-1}\intl_{\hn} f(y)\, \frac{([x,y] -1)^{-1/2}}{([x,y] +1)^{n/2-1}}\, dy\nonumber\\
&&=\z_{n,n-1}\intl_0^\infty  \frac{(\ch r \!-\!1)^{-1/2}}{(\ch r\! +\!1)^{n/2-1}}\,\sh^{n -1} r \, dr
\!\!\intl_{\sn} \!\! f(\theta\, \sh r  \!+\! e_{n+1}  \ch r) \, d \theta\nonumber\\
&&\label {MMM-HYP3horT2}=c_2 \intl_1^\infty (t-1)^{(n-3)/2}\,dt  \intl_{\sn}  f_x (\sqrt {t^2 -1}\, \th +t\,e_{n+1})\, d\theta, \eea
\[
c_2=\frac{\pi^{(1-n)/2}}{2^{(n+1)/2}\,\Gam ((n-1)/2)}.\]
Comparing (\ref {MMM-HYP3horT1}) and (\ref {MMM-HYP3horT2}), we obtain (\ref{taRRRrayhor}).
\end{proof}

We will need  an analogue of Lemma \ref{duJJYPhor} for $\a=1$. Starting from (\ref{mm983XX2}) with $i=1$, we
 define
\be\label{tag 4.6OFR1HOOO8}
(\stackrel{*}{\fr{H} }{\!}^1 \vp)(x)=\gam'_n\intl_{\Gam_+} \vp (\xi)\, \log \left |\frac {[x,\xi]-1}{[x,\xi]^{1/2}}\right | \, \frac{d\xi}{[x,\xi]^{n/2}},\ee
where
\[\gam'_n=\lim\limits_{\a \to 1}  (\a -1)\,\gam_\a=-\frac{\pi^{-n/2}}{2^{n-1}\,\Gam (n/2)}.\]

\begin{proposition} \label{duJJYP1} Let $\vp = \fr{H} f$, $f \in C^\infty_c(\hn)$, $n\ge 2$. Then
\be\label{tag 4.6OFR1HOOO}
\stackrel{*}{\fr{H}}{\!}^1 \vp = Q^{n} f+ \Phi,\ee
where
\be\label{tMJOO8}
\Phi (x)=\tilde \gam\,\intl_{\hn} f(y)\frac{dy}{([x,y]\!+\!1)^{n/2 -1}}=2^{1-n/2}\tilde \gam\,\intl_{\Gam_+} \vp (\xi)\, \frac{d\xi}{[x,\xi]^{n/2}}\, ,\ee
\[\tilde \gam =\frac{\psi(n/2)-\psi (1/2)-\log 2}{\pi^{n/2}\,2^{n/2+1}\,\Gam (n/2)}, \qquad \psi(z)=\frac{\Gam' (z)}{\Gam (z)}.\]
\end{proposition}
\begin{proof}
  For  $\a  \neq 1$, but close to $1$, we can write the equality
 \[  \gam_\a \intl_{\Gam_+}\! (\fr{H}  f)(\xi)\, \frac{|[x,\xi]\!-\! 1|^{\a-1}}{[x,\xi]^{(n+\a-1)/2}}\, d\xi= (Q^{\a+n-1} f)(x)\]
  (cf. (\ref{mm983XX1} )  as
\bea  &&\gam_\a \intl_{\Gam_+} \vp (\xi)\,\left [\left | \frac{[x,\xi]-1}{[x,\xi]^{1/2}}\right |^{\a -1}\!-\!1\right ] \frac{d\xi}{[x,\xi]^{n/2}}   + \gam_\a \,I_1\nonumber\\
&&=\z_{n,\a+n-1} \intl_{\hn}  f(y)\, \frac{([x,y]-1)^{(\a -1)/2}\! -1}{([x,y]+1)^{n/2-1}}\, dy +\z_{n,\a+n-1}\,I_2,\nonumber\eea
where
\[I_1=\intl_{\Gam_+} \vp (\xi)\, \frac{d\xi}{[x,\xi]^{n/2}}, \qquad I_2=\intl_{\hn}  f(y)\, \frac{dy}{([x,y]+1)^{n/2-1}},\]
\[ \gam_\a=\frac{ \pi^{(1-n)/2}\,\Gam ((1-\a)/2)}  {2^{\a+n-1}\,\Gam (n/2)\, \Gam (\a/2)},\] \[ \z_{n,\a+n-1}=\frac{\Gam ((1-\a)/2)}{2^{(\a+n+1)/2}\,\pi^{n/2}\,\Gam ((\a+n-1)/2)}\,.\]
By (\ref{duUII2a}), we have $I_1=2^{n/2-1} I_2$.
Hence,
\bea  &&\gam_\a \intl_{\Gam_+} \vp (\xi)\,\left [\Big| \frac{[x,\xi]-1}{[x,\xi]^{1/2}}\Big|^{\a -1}\!-\!1\right ]
\frac{d\xi}{[x,\xi]^{n/2}}   \nonumber\\
\label{tHOOOFR1}&&=\z_{n,\a+n-1} \intl_{\hn}  f(y)\, \frac{([x,y]-1)^{(\a -1)/2}\! -1}{([x,y]+1)^{n/2-1}}\, dy  + \tilde \gam_\a \,I_2,\eea
where
 $\tilde \gam_\a=\z_{n,\a+n-1}- 2^{n/2-1}\, \gam_\a$.    Passing to the limit  as $\a \to 1$, we obtain
$\stackrel{*}{\fr{H}}{\!}^1 \vp = Q^{n} f+\tilde \gam\, I_2=Q^{n} f+2^{1-n/2}\tilde \gam\, I_1$, as desired.
\end{proof}

Propositions \ref {Corollary 4.5HY} and \ref {duJJYP1} combined with the properties of the potential type operator $Q^\a$ (see Section \ref {ExaSinreQ} ) give the following inversion result for the horospherical transforms.

\begin{theorem}\label{ThHORYP}  Let $\varphi = \fr{H}  f$, $f \in C^\infty_c(\hn)$,
\be\label {tag 1HORnYO} \P_\ell (\Delta_H) = (-1)^{\ell}\prod\limits^\ell_{i=1} \big[\Delta_H + i(n-1-i)\big], \qquad  \ell \in\bbn.\ee

\noindent {\rm (i)}  \ If $n$ is odd, then for   $\ell= (n-1)/2$,
\be\label {tag 1.6hnYO} f=\lam_n \,\P_{(n-1)/2}(\Delta_H)\,\fr{H}^* \varphi, \qquad  \lam_n=\frac{2^{1-n}\,  \pi^{1-n/2}}{\Gam (n/2)}.\ee

\noindent {\rm (ii)} \ If  $n=2$, then

\be\label {tag 1HORnYOA} f= -\Delta_H \stackrel{*}{\fr{H}}{\!}^1 \vp +\frac{1}{4\, \pi}\intl_{\Gam_+} \vp (\xi)\, d\xi.\ee

\noindent {\rm (iii)} \ If  $n=4,6, \ldots $, then
\be\label {tag 1HORnYOA} f= \P_{n/2}(\Delta_H) \stackrel{*}{\fr{H}}{\!}^1 \vp, \quad \P_{n/2}(\Delta_H) = (-1)^{n/2} \prod\limits_{i=1}^{n/2}   (\Delta_H+i (n-1-i)).  \ee
\end{theorem}
\begin{proof}  If $n$ is odd, then (\ref{taRRRrayhor}) gives $\fr{H}^* \varphi=\lam_n^{-1}\, Q^{n-1} f$ and (\ref{tag 1.6hnYO}) follows from Proposition \ref {Corollary 4.5HY}. If $n=2$, then, by (\ref{tag 4.6OFR1HOOO}),
\[\stackrel{*}{\fr{H}}{\!}^1 \vp = Q^{2} f+ \Phi, \qquad \Phi=\tilde \gam \intl_{\hn} f(y)\,dy\equiv \const.\]
Applying  $-\Delta_H$ to both sides of this equality, owing to (\ref{tYUY343X}), we obtain
\bea
-\Delta_H \stackrel{*}{\fr{H}}{\!}^1 \vp&=& - \Delta_H Q^2 f - \Delta_H \Phi=- \Delta_H Q^2 f=f-\frac{1}{4\, \pi}\intl_{\bbh^2} f(y)\, dy\nonumber\\
&=&f-\frac{1}{4\, \pi}\intl_{\Gam_+} \vp (\xi)\, d\xi;\nonumber\eea
cf. (\ref{duUII2a}). This gives (\ref{tag 1HORnYOA}).
If $n\ge 4$, then, by (\ref{tag 4.6OFR1HOOO}) and (\ref{tYUY341}),
\[
\P_{n/2}(\Delta_H) \stackrel{*}{\fr{H}}{\!}^1 \vp=\P_{n/2}(\Delta_H) Q^n f +\P_{n/2}(\Delta_H) \Phi =f + (\tilde \gam/\z_n')\P_{n/2}(\Delta_H)\, Bf.\]
By Lemma \ref{Le76768QQ1}, $\P_{n/2}(\Delta_H)\, Bf=D_2\cdots D_{n-2}Bf=0$. Hence, we are done.
\end{proof}


\begin{thebibliography}{[ASMR]}


\bibitem  {BC94}    C. A. Berenstein and E. Casadio Tarabusi.   \textit{An inversion formula for the
horocycle transform on the real hyperbolic space, } Lectures in Appl.
Mathematics 30 (1994),  1--6.


\bibitem  {Bray94}  W. O. Bray. Aspects of harmonic analysis on real hyperbolic space. In {\it Fourier analysis: analytic and geometric aspects}, ed. by W. O. Bray, P. S. Milojevic, and \v{C}aslav V. Stanojevi\'{c}, Lect. Notes Pure Appl. Math.  157.  Marcel Dekker, 1994, pp. 77--102.

\bibitem   {Bray96}  \bysame. Generalized spectral projections on symmetric spaces of non-compact type: Paley-Wiener theorems. \textit{Jour. Funct. Anal.} {\bf 135} (1996), 206--232.

\bibitem  {Bru} W. O. Bray and B. Rubin. Inversion of the horocycle transform on real hyperbolic spaces via wavelet-like transforms. In  {\it Analysis of divergence: control and management of divergent processes}, ed. by W. O. Bray and C. V. Stanojevic, Birkhauser, 1999, 87--105.

%\bibitem  {BS88} W. O. Bray and D. C. Solmon. The horocycle transform and harmonic analysis on the Poincar\'{e} disk. Preprint (1988).

\bibitem   {BS90} W. O. Bray and D. C. Solmon.  Paley-Wiener theorems on rank one symmetric spaces of non-compact type. \textit{Contemp. Math.} {\bf113} (1990), 17--29.

\bibitem {CFKP} J. W. Cannon, W. J. Floyd, R. Kenyon, and W. R. Parry. Hyperbolic geometry. In \textit{Flavors of geometry}. Math. Sci. Res. Inst. Publ. 31. Cambridge Univ. Press, Cambridge, 1997, 59-–115.

\bibitem  {ClS} J. L. Clerc and E. M. Stein. $L^p$-multipliers for noncompact symmetric spaces. \textit{Proc. Nat. Acad. Sci. USA} {\bf 71} (1974),  3911--3912.


\bibitem {Er} A. Erd\'elyi (Editor),  Higher transcendental functions, Vol. I
and  II, McGraw-Hill, New York, 1953.


\bibitem {Far79} J. Faraut. Distributions sph\'eriques sur les espaces hyperboliques.  \textit{J. Math. Pures Appl.} (9) \textbf{58} (1979), 369–-444.

 \bibitem {Far82} \bysame. Analyse harmonique sur les paires de Guelfand et les espaces hyperboliques, in  \textit{Analyse harmonique,  Les Cours du CIMPA} 1982, 315-–446.

\bibitem {Far83} \bysame.  Analyse harmonique sur les espaces hyperboliques. [Harmonic analysis on hyperbolic spaces] In \textit{Topics in modern harmonic analysis, Vol. I, II (Turin/Milan, 1982).} Ist. Naz. Alta Mat. Francesco Severi, Rome, 1983, pp. 445-–473.



\bibitem  {FJK73}  M. Flensted-Jensen  and  T. Koornwinder.  The convolution structure for Jacobi function expansions. \textit{Ark. Mat.}
{\bf 11} (1973), 245--262.

\bibitem {F} B. Fuglede.  An integral formula. \textit{Math. Scand.} \textbf{6} (1958), 207--212.


\bibitem {GGV} I. M. Gelfand, M. I. Graev, and N. J. Vilenkin. \textit{Generalized functions, Vol 5. Integral geometry and representation theory}, Academic Press,  1966.

\bibitem {Gi01} S. G. Gindikin.
Integral geometry on hyperbolic spaces. In {\it Harmonic analysis and integral geometry} (Safi, 1998),
 41-–46, Chapman \& Hall/CRC Res. Notes Math., 422, Chapman \& Hall/CRC, Boca Raton, FL, 2001.


\bibitem {Gi08} \bysame.  Horospherical transform on  Riemannian symmetric manifolds of noncompact type, \textit{Funct. Anal. Appl.} (4) \textbf{42} (2008), 290-–297.


\bibitem {Gi13}  \bysame. Local inversion formulas for horospherical transforms. \textit{Mosc. Math. J.} (2) \textbf{13} (2013), 267–-280, 363.


\bibitem   {Go10a}  F. B. Gonzalez.  Conical distributions on the space of flat horocycles. \textit{J. Lie Theory} (3) \textbf{20} (2010), 409–-436.

\bibitem {GQ94} F. B. Gonzalez and E. T. Quinto. Support theorems for Radon transforms on higher rank symmetric spaces.
\textit{Proc. Amer. Math. Soc. } (4) \textbf{122} (1994), 1045–-1052.



\bibitem {H73}  S. Helgason.  The surjectivity of invariant differential operators on symmetric spaces. I. \textit{Ann. of Math.} (2) \textbf{98} (1973), 451--479.

\bibitem {H00} \bysame.  \textit{Groups and geometric analysis: integral geometry,
invariant differential operators, and spherical functions}. Academic Press,   2000.


\bibitem {H08} \bysame. \textit{Geometric analysis on symmetric spaces, Second Edition}.  Amer. Math. Soc.,  Providence, RI, 2008.

\bibitem {H11}  \bysame.  \textit{Integral geometry and Radon transform}. Springer, New York-Dordrecht-Heidelberg-London, 2011.

\bibitem {H12} \bysame. Support theorems for horocycles on hyperbolic spaces. \textit{Pure Appl. Math. Q.} (4) \textbf{8} (2012), 921-–927.

\bibitem {HR} E. Hewitt and K. A. Ross.   \textit{Abstract harmonic analysis, Vol. I}.  Springer, Berlin, 1963.


\bibitem {HPVa}  J. Hilgert, A. Pasquale, and E. B. Vinberg. The dual horospherical Radon transform for polynomials. \textit{Mosc. Math. J.} (1) \textbf{2} (2002), 113–-126, 199.

\bibitem {HPV}  \bysame.  The dual horospherical Radon transform as a limit of spherical Radon transforms.
In \textit{Lie groups and symmetric spaces. Amer. Math. Soc. Transl. Ser. 2, 210}. Amer. Math. Soc., Providence, RI, 2003, pp. 135–-143.



\bibitem {Kat05}  A. Katsevich. An inversion formula for the dual horocyclic Radon transform on the hyperbolic plane. Math. Nachr.  \textbf{278} (2005), no. 4, 437–- 450.

\bibitem {Liz93} P. I. Lizorkin.   Direct and inverse theorems of approximation
theory for functions on Lobachevsky space. \textit{Proc. of the Steklov Inst. of Math.} \textbf{4} (1993), 125--151.

\bibitem {Mol66} V. F. Molchanov. Harmonic analysis on a hyperboloid of one sheet [in Russian]. \textit{Dokl. Akad. Nauk SSSR} \textbf{171} (1966), 794–-797.

\bibitem {Mol76} \bysame. Spherical functions on hyperboloids, [in Russian]. \textit{Mat. Sb. (N.S.)} (2) \textbf{99(141)} (1976), 139–-161, 295.

%\bibitem {Mol80} \bysame.  Plancherel's formula for hyperboloids, [in Russian]. \textit{Boundary value problems of mathematical physics, 10. %Trudy Mat. Inst. Steklov.} \textbf{147} (1980), 65–-85, 203.

\bibitem  {Pet93} I. V.  Petrova.  Approximation on a hyperboloid in the $L_2$ metric, [in Russian]. \textit{Trudy Mat. Inst. Steklov. 194 (1992), Issled. po teor. differ. funktsii mnogikh peremen. i ee prilozh.} \textbf{14} (1992), 215--228; translation in \textit{Proc. Steklov Inst. Math.} (4) \textbf{194} (1993), 229–-243.

\bibitem  {PBM1} A. P. Prudnikov, Y. A. Brychkov, and  O. I. Marichev. \textit{Integrals and series: elementary functions.} Gordon and Breach Sci. Publ., New York-London, 1986.

%\bibitem  {PBM2} \bysame. \textit{Integrals and Series:  Special Functions.} Gordon and Breach Sci. Publ., New York-London, 1986.

\bibitem  {PBM3} A. P. Prudnikov, Y. A. Brychkov, and  O. I. Marichev. \textit{Integrals and series:  supplementary chapters.}  Gordon and Breach Sci. Publ., New York, 1990.

\bibitem  {Ross} W. Rossmann. Analysis on real hyperbolic spaces. \textit{J. Funct. Anal.} (3) \textbf{30} (1978), 448-–477.

\bibitem  {Rou} F. Rouvi\`{e}re.  Inverting Radon transforms: the group-theoretic approach.
\textit{Enseign. Math.} (3-4) \textbf{(2) 47} (2001), 205--252.


\bibitem  {Ru02a} B. Rubin.  Helgason-Marchaud inversion formulas for Radon transforms. \textit{Proc. Amer. Math. Soc.} \textbf{130} (2002),
 3017--3023.

\bibitem  {Ru02b}  \bysame. Inversion formulas for the spherical  Radon transform and the generalized cosine transform.  \textit{Advances in Appl. Math.} \textbf{29} (2002), 471--497.


\bibitem {Ru02c} \bysame. Radon, cosine, and sine transforms on real hyperbolic space. \textit{Advances in Math.} \textbf{170} (2002), 206--223.


\bibitem  {Ru13b} \bysame.  On the Funk-Radon-Helgason inversion method in integral geometry. \textit{Cont. Math.}  \textbf{599} (2013), 175--198.

\bibitem  {Ru13c} \bysame. Semyanistyi fractional integrals and Radon transforms. \textit{Cont. Math.} \textbf{598} (2013), 221--237.


\bibitem  {Se1} V. I. Semyanistyi. On some integral transformations in Euclidean space [in Russian]. \textit{Dokl. Akad. Nauk SSSR,}   \textbf{134} (1960),  536--539.


%\bibitem  {ST78} R. J. Stanton and P. A. Tomas. Expansions for spherical functions on noncompact symmetric spaces. \textit{Acta Math.}
%\textbf{140} (1978),  251--276.

%\bibitem  {Strom} J.-O. Str\"omberg. Weak type $L^1$ estimates for maximal functions  on noncompact symmetric spaces. \textit{Ann. of Math.} {\bf 114} (1981), %115--126.

\bibitem  {V} N. Ja. Vilenkin.  \textit{Special functions and the theory of group representations, Translations of Mathematical Monographs, Vol. 22.} American Mathematical Society, Providence, R. I. 1968.


\bibitem  {VK} N. Ja. Vilenkin and A. V. Klimyk.  \textit{Representations of Lie groups and special functions,
Vol. 2.}  Kluwer Academic Publishers,  Dordrecht, (1993).

\bibitem  {Zo1}  A. V. Zorich.  Inversion of integral transformations connected with nilpotent subgroups of complex semisimple Lie groups. (Russian) \textit{Algebra i Analiz} \textbf{2} (1990), no. 1, 73--113; translation in Leningrad Math. J. \textbf{2} (1991), no. 1, 65--96.   

\bibitem  {Zo2}   \bysame. Inversion of horospherical integral transform on real semisimple Lie groups. \textit{Infinite analysis, Part A, B (Kyoto, 1991)}, 1047--1071, Adv. Ser. Math. Phys., \textbf{16}, World Sci. Publ., River Edge, NJ, 1992.


\end{thebibliography}
\end{document}